\documentclass[oneside]{amsart}

\usepackage{geometry}
\usepackage{amssymb}
\usepackage[mathscr]{euscript}
\usepackage{amsmath}
\usepackage{mathtools}
\usepackage{enumitem} 
\usepackage{mathrsfs}
\usepackage{lipsum}
\usepackage{amsfonts}
\usepackage{graphicx}
\usepackage{epstopdf}
\usepackage{algorithmic}
\usepackage{hyperref}
\usepackage{tikz-cd}

\newcommand{\Aut}{\operatorname{Aut}}
\newcommand{\End}{\operatorname{End}}
\newcommand{\Pol}{\operatorname{Pol}}

\newcommand{\Th}{\operatorname{Th}}

\renewcommand{\int}{\operatorname{int}}

\newcommand{\CSP}{\operatorname{CSP}}

\newcommand{\eq}{\mathrm{eq}}

\newcommand{\cA}{\mathcal A}
\newcommand{\cB}{\mathcal B}

\newcommand{\cS}{\mathcal S}

\newcommand{\bA}{\mathbb A}
\newcommand{\bB}{\mathbb B}
\newcommand{\bC}{\mathbb C}
\newcommand{\bD}{\mathbb D}
\newcommand{\bQ}{\mathbb Q}
\newcommand{\bG}{\mathbb G}
\newcommand{\bR}{\mathbb R}
\newcommand{\bN}{\mathbb N}
\newcommand{\bE}{\mathbb E}

\usepackage{thm-restate}

\theoremstyle{plain}
\newtheorem{theorem}{Theorem}
\newtheorem{lemma}[theorem]{Lemma}
\newtheorem{corollary}[theorem]{Corollary}
\newtheorem{proposition}[theorem]{Proposition}
\newtheorem{conjecture}{Conjecture}

\theoremstyle{definition}
\newtheorem{definition}[theorem]{Definition}
\newtheorem{example}[theorem]{Example}

\theoremstyle{remark}
\newtheorem*{remark}{Remark}

\title{Decidability of Interpretability} 

\thanks{Funded by the European Union (ERC, POCOCOP, 101071674). Views and opinions expressed are however those of the author(s) only and do not necessarily reflect those of the European Union or the European Research Council Executive Agency. Neither the European Union nor the granting authority can be held responsible for them. This research was funded in whole or in part by the Austrian Science Fund (FWF) [I 5948]. For the purpose of Open Access, the authors have applied a CC BY public copyright licence to any Author Accepted Manuscript (AAM) version arising from this submission.}

\author{Roman Feller}
\author{Michael Pinsker}

\begin{document}

\begin{abstract}
    The Bodirsky-Pinsker conjecture asserts a P vs. NP-complete dichotomy for the computational complexity of Constraint Satisfaction Problems (CSPs) of first-order reducts of  finitely bounded homogeneous structures. Prominently, two structures in the scope of the conjecture have log-space equivalent CSPs if they are pp-bi-interpretable, or equivalently, if their polymorphism clones are topologically isomorphic. The latter gives rise to the algebraic approach which regards structures with topologically isomorphic polymorphism clones as equivalent and seeks to identify structural reasons for hardness or tractability in topological clones. We establish that the equivalence relation of pp-bi-interpretability underlying this approach is reasonable: On the one hand, we show that it is decidable under mild conditions on the templates; this improves a theorem of Bodirsky, Pinsker and Tsankov (LICS'11) on decidability of equality of polymorphism clones. On the other hand, we  show that within the much larger class of transitive $\omega$-categorical structures without algebraicity, the equivalence relation is of lowest possible complexity in terms of  descriptive set theory: namely, it is smooth, i.e., Borel-reduces to equality on the real numbers. 
    On our way to showing the first result, we establish that the model-complete core of a structure that has a finitely bounded homogeneous Ramsey expansion (which might include all structures of the Bodirsky-Pinsker conjecture)  is computable, thereby providing a constructive alternative to  previous non-constructive proofs of its existence.
\end{abstract}

\maketitle

\section{Introduction} 

For a countable relational structure $\bA$ in a finite relational language, the constraint satisfaction problem with template $\bA$, for short $\CSP(\bA)$, denotes the computational problem of deciding whether or not a given finite relational structure homomorphically maps into $\bA$. This 
provides a very flexible framework capable of  modelling, up to polynomial-time Turing reductions,  every computational problem one dares to dream about~\cite{BodirskyGrohe}. A typical example is the $3$-colouring problem of graphs, modelled as the CSP of the complete loopless graph on three vertices. As a further example,  the computational problem of checking if a finite input graph is triangle-free can be modelled as the CSP of  Henson's  (infinite) generic triangle-free graph. This problem is provably not a finite-domain CSP, i.e., there does not exist any finite graph $\bG$ such that $\CSP(\bG)$ describes said problem. This demonstrates the necessity 
 to also consider 
templates with an infinite domain, at least if one wishes to capture some of the most natural CSPs such as vertex-colouring problems with forbidden patterns (CSPs expressible in MMSNP~\cite{MMSNPFederVardi}), edge-colouring problems with forbidden patterns (CSPs expressible in GMSNP~\cite{GMSNP}), and many more.

As of today, however, the computational complexity of finite-domain CSPs is much better understood. A deep algebraic theory using the language of universal algebra has been developed for this situation after Feder and Vardi had conjectured a P vs. NP-complete complexity dichotomy for CSPs  with a finite domain, culminating in the independent confirmation of this dichotomy in 2017, by Bulatov and Zhuk~\cite{BulatovFVConjecture,Zhuk17,Zhuk20}. In the infinite world such a complexity dichotomy is simply not possible even if one restricts to templates that are $\omega$-categorical, a model-theoretic notion rendering countable structures ``close to being finite''. Formally, $\omega$-categoricity of a structure means that its automorphism group acts with finitely many orbits in the componentwise action on its $n$-tuples, for all $n\geq 1$; intuitively, this means that the structure only distinguishes finitely many $n$-tuples up to automorphisms, and a typical example is  the order of the rational numbers $(\mathbb Q,<)$. While this mathematical property makes an algebraic approach to its CSP possible  as we will explain further in Section~\ref{sect:polymorphisms}, it has in itself little bearing on the computational tameness of the CSP~\cite{GJKMP-conf, BodirskyGrohe}.

There is however, a class of $\omega$-categorical structures, \emph{first-order reducts of finitely bounded homogeneous structures}, giving rise to CSPs within NP while substantially generalising the class of computational problems that can be modelled as finite domain CSPs. For example, in this class we can model problems arising in temporal and spatial reasoning~\cite{BodirksyAI}, as well as the aforementioned colouring problems. The class is defined as follows. A homogeneous structure in a finite relational language satisfies a particularly strong form of $\omega$-categoricity, where not only the number of orbits of $n$-tuples is finite, but where the orbit of each tuple is completely determined by the relations that hold between its components (for which there are only finitely many possibilities). Again, the structure $(\mathbb Q,<)$ is a good example: the three possible orbits of a pair $(a,b)$ are given by $a=b$, $a<b$, and $b<a$. 
Finite boundedness means that moreover, the orbits that appear in the structure are described by a finite set of ``forbidden patterns''. In the example of $(\mathbb Q,<)$, no triple $(a,b,c)$ satisfies $a<a$ or $a<b<c<a$; it is easy to come up with a finite set of forbidden patterns that suffice to define the possible relations that can hold on a tuple. Finally, $\bC$ being a first-order reduct of a finitely bounded homogeneous structure $\bA$ means that each relation of $\bC$ is a union of orbits of $\bA$ of the same arity. This is the same as $\bC$ being first-order definable in $\bA$, or equivalently, since finitely bounded homogeneous structures have quantifier elimination, being definable in $\bA$ by quantifier-free formulas.
 
It is easy to see that if $\bC$ is a first-order reduct of a finitely bounded homogeneous structure $\bA$, then its CSP is in NP: an input structure can be enumerated to obtain a tuple, and relations for the language of $\bA$ can be randomly ``guessed'' on the tuple; it can then be verified in polynomial time if this assignment yields an orbit of $\bA$ and whether or not it yields a homomorphism into $\bC$.
\begin{conjecture}[Bodirsky and Pinsker $\sim$2011, see~\cite{BPP-projective-homomorphisms}]\label{conj:bp}
    Let $\bC$ be a first-order reduct of a finitely bounded homogeneous structure $\bA$. Then $\CSP(\bC)$ is either in P or NP-complete.
\end{conjecture}

While the conjecture has been confirmed in many instances such as first-order reducts of $(\mathbb Q,<)$~\cite{tcsps-journal}, a countable set with equality~\cite{ecsps}, the universal homogeneous poset~\cite{posetCSP18}, homogeneous undirected graphs~\cite{BMPP16}, in particular the random graph~\cite{BodPin-Schaefer-both}, the universal homogeneous tournament~\cite{MottetPinskerSmooth} as well as MMSNP~\cite{MMSNP_conf,MMSNP-Journal}, and graph orientation problems with forbidden tournaments~\cite{forbiddenTournaments,AntoineZenoPaper,FellerPinsker}, it is still wide-open in general.  

The common toolbox for finite-domain CSPs and
more generally 
CSPs within the Bodirsky-Pinsker conjecture is the analysis of polymorphisms, which are certain algebraic invariants of the solution space of a CSP; we shall discuss them next. 

\subsection{Polymorphisms and the complexity of a CSP}\label{sect:polymorphisms}

A polymorphism of a relational structure $\bA$ is a homomorphism from a finite power of $\bA$ into itself. The collection of these ``symmetries of higher arity'', denoted by $\Pol(\bA)$, forms an algebraic object called a clone. 
The computational complexity of the CSP of an $\omega$-categorical structure is tightly connected with its polymorphism clone. 
In~\cite{BodirskyNesetrilJLC} Bodirsky and Ne\v{s}et\v{r}il show that two $\omega$-categorical structures with the same polymorphism clones have polynomial-time equivalent CSPs.  
Thus, if one is interested in the  complexity of the CSP of such structures up to polynomial-time equivalence, two structures may be regarded as equivalent if their polymorphism clones are equal. 

The underlying mathematical content of this reduction is that the
equivalence of having equal polymorphism clones  admits a description using a fragment of first-order logic: primitive positive (pp) formulas. A first-order formula is said to be primitive positive if it is of the form $\exists y_1\cdots \exists y_k \phi$, where $\phi$ is a conjunction of atomic formulas. Namely, in~\cite{BodirskyNesetrilJLC} it is shown that two structures $\bA$ and $\bB$ have the same polymorphism clones precisely when they pp-define the same sets of relations, i.e., when they are \emph{pp-interdefinable}. This insight shaped the early days, the first regime, of infinite-domain CSP research and led to complexity dichotomies for first-order reducts of, e.g., a countable set with equality~\cite{ecsps} and reducts of the rational linear order $(\bQ,<)$~\cite{tcsps-journal}. Such classifications were obtained by first fixing a finitely bounded homogeneous ``ground structure'' (e.g.,~$(\mathbb Q,<)$) and then considering all possible polymorphism clones of its first-order reducts; these are precisely the polymorphism clones containing 
the automorphism group of the ground structure. 

Moving to the meta-perspective, the equivalence of pp-inter\-definability underlying this first regime was investigated from a computational viewpoint in~\cite{BPT-decidability-of-definability}. 
More precisely it is shown that it is decidable to determine whether two first-order reducts $\bC,\bD$ of a finitely bounded homogeneous  structure $\bA$ are pp-interdefinable, provided that $\bA$ is a Ramsey structure. While this latter condition of combinatorial nature might seem like a strong restriction at first, this need not be the case whatsoever: 
it is an open question whether every first-order reduct of a  finitely bounded homogeneous structure (i.e.~a structure within the scope of Conjecture~\ref{conj:bp}) is also a first-order reduct of a  finitely bounded homogeneous Ramsey structure, see~\cite{BPT-decidability-of-definability};~\cite{RamseyClasses}, Conjecture 1.1. From a practical point of view, all methods available today depend on the ground structure $\bA$ being Ramsey; see, e.g.,~the survey~\cite{pinsker-sheep}. 
We remark that the decidability statement indeed makes sense since finitely bounded homogeneous structures are uniquely determined by their finite set of finite bounds, and hence can be fed to an algorithm; a first-order reduct is then given to the algorithm in the form of quantifier-free formulas defining its relations.

\begin{theorem}[Bodirsky, Pinsker, and Tsankov~\cite{BPT-decidability-of-definability}]\label{thm:DecOfDef}
    The following problem is decidable:
    \begin{itemize}
    \item \textbf{Given:} A finitely bounded homogeneous Ramsey structure $\bA$ and first-order reducts $\bC,\bD$ of it.
    \item \textbf{Question:} Are $\bC$, $\bD$ pp-interdefinable?
    \end{itemize}
\end{theorem}

We remark that as a subset of their proof, the authors obtain the same statement for existential positive (ep-)interdefinability; on the other hand, deciding first-order interdefinability, stated as an open problem in~\cite{BPT-decidability-of-definability}, remains open to this day. 

Vastly generalizing the mentioned result on pp-interdefinability, Bodirsky and Pinsker~\cite{Topo-Birk} later showed that in fact the complexity of a CSP, up to log-space reductions, is already determined by the polymorphism clone up to topological clone isomorphism. Similarly to the previous situation with equality of polymorphism clones, the mathematical content underlying this reduction admits a description using pp-formulas, this time using the concept of \emph{pp-bi-interpretability} instead of pp-interdefinability.  A pp-interpretation of a structure $\bB$ in $\bA$ is a partial surjective function $\Gamma$ from a finite power of $\bA$ onto $\bB$ so that preimages of pp-definable sets remain pp-definable. It is a pp-bi-interpretation if there additionally exists a pp-interpretation $\Delta$ of $\bA$ in $\bB$ that is, in some sense, inverse to $\Gamma$. 

This result of Bodirsky and Pinsker allows, in principle,  for a more structured approach to CSPs within Conjecture~\ref{conj:bp}: rather than considering each polymorphism clone one-by-one (say, for all fo-reducts $\bC$ of a fixed finitely bounded homogeneous  structure $\bA$, as was done previously in the above-mentioned first-generation  classifications), develop a theory for the possible structure of polymorphism clones, viewed as topological clones. This direction was first taken in a CSP dichotomy for problems expressible in the logic MMSNP~\cite{MMSNP_conf,MMSNP-Journal}, and culminated in the  technique of ``smooth approximations''~\cite{MottetPinskerSmoothConf,MottetPinskerSmooth}, which was successfully used to essentially reprove all known CSP complexity dichotomies for templates within Conjecture~\ref{conj:bp} to this day. Such classifications have been called second-generation proofs in~\cite{MottetPinskerSmoothConf,MottetPinskerSmooth}.

The new philosophy means precisely that the equivalence relation on structures of being pp-bi-interpretable (rather than the finer equivalence of being pp-interdefinable) lies at the heart of modern-day infinite-domain CSP research. Therefore, showing the reasonability of this equivalence would justify the validity of this approach.
In particular it would be reassuring to know that, just like pp-interdefinability, pp-bi-interpretability is decidable, and this is the central theme of the present work.

\subsection{Descriptive set theory} \label{sect_intro_descriptive}
 While the above-mentioned theorem of Bodirsky and Pinsker on the equivalence of  topological isomorphism of polymorphism clones with pp-bi-interpretability is useful for a structured approach to Conjecture~\ref{conj:bp}, it actually holds in the larger class of all $\omega$-categorical structures. A decidability result like Theorem~\ref{thm:DecOfDef}, however, does not  make sense for this class, as there are uncountably many non-isomorphic $\omega$-categorical structures and we therefore cannot encode them as inputs of an algorithm. Therefore, if we wish to measure how reasonable, or how complex the equivalence of bi-interpretability on the class of $\omega$-categorical structures is, we have to resort to other, non-computational measures; here, descriptive set theory turns out to provide the adequate language.

Namely, it provides a framework to compare the complexity of classification problems encountered throughout mathematics. Such problems can often be described by equivalence relations on Polish spaces, and \emph{Borel-reducibility} provides a robust notion to compare their complexity~\cite{Kechris_2024}. It is defined as follows, see Section~\ref{sect:Prelims} for definitions of the underlying notions. Let $(X,E_X)$ and $(Y,E_Y)$ be Polish spaces (think $\bR$ with the Euclidean topology) equipped with equivalence relations respectively. 
The relation $E_X$ Borel-reduces to $E_Y$, denoted by $(X,E_X) \leq_B (Y,E_Y)$ if there is a Borel map $f\colon X \to Y$ that reduces $E_X$ to $E_Y$, i.e., for all $x,y \in X$ it holds
\begin{equation*}
    x \mathrel{E_X} y \iff f(x) \mathrel{E_Y} f(y).
\end{equation*}
Note the analogy to many-one reductions: the function $f$ is ``nice'' (it is Borel) and allows us to reduce the question  whether $x,y$ are equivalent in $E_X$ to the question whether $f(x), f(y)$ are equivalent in $E_Y$. 
In this sense, one can say that the classification problem of $E_X$ is at most as difficult as the one of $E_Y$. We write $E_X <_B E_Y$ if $E_X \leq_B E_Y$ and $E_Y \not \leq_B E_X$.
This notion of reducibility was introduced in the seminal paper~\cite{HKLBorel} by Harrington, Kechris, and Louveau, and it provides a framework for comparing the intrinsic complexity of classification problems. Prominent examples of this approach are the classifications of countable torsion-free abelian groups up to isomorphism~\cite{TorsionFreeGroups}, the classification of separable $C^*$-algebras~\cite{C*Algebras}, and the 
classification of ergodic measure preserving transformations~\cite{ErgodicClassifiaction}. 

Work of Silver~\cite{Silver} and Harrington-Kechris-Louveau~\cite{HKLBorel} shows that Borel equivalence relations on Polish spaces have the following initial segment under Borel-reducibility 
\begin{equation*}
    1 <_B 2 <_B \cdots <_B \bN <_B (\mathbb R,=). 
\end{equation*}
Here each natural number $n$ stands for a Borel equivalence relation with $n$-many equivalence classes, $\bN$ stands for Borel equivalence relations with countably many classes and $(\mathbb R,=)$ simply stands for equality on the real numbers. Any Borel equivalence relation is either equivalent to one of the above, and in this case is called \emph{smooth}, or strictly lies above $\bR$ under Borel-reducibility. In this sense smooth equivalence relations are the most basic, ``tractable'', equivalence relations from a Borel-reducibility point of view. 

CSP templates, i.e., countable structures in a finite language, can naturally be encoded as (countable) 0-1-sequences, and hence as elements of the real numbers $\mathbb R$. The subclass of $\omega$-categorical structures  form a natural Polish space (a subspace of $\mathbb R$), and thus the complexity of any equivalence relation on this class can be evaluated in terms of Borel reducibility. In particular, this is true for the equivalence pp-bi-interpretability, and smoothness of this relation would be a strong indication for its reasonability as a foundation for a structured algebraic approach to the CSPs of this class.

\subsection{Model-complete cores}\label{sect:cores}
In our quest to understand the complexity of pp-bi-interpretability, we will naturally come across a central concept of the algebraic approach to CSPs, namely that of a (model-complete) core. We shall now explain this notion.   
An easy, yet important insight for finite-domain CSPs is that for any finite relational structure $\bA$ there is a unique structure $\bA'$ of smallest domain size with the same CSP, called the \emph{core} of $\bA$. 
This can be generalized to $\omega$-categorical structures,
where the notion of minimal size is replaced to mean that all endomorphisms of $\bA'$ locally resemble, i.e., restricted to finite sets are equal to, (restrictions of) automorphisms. The existence of these \emph{model-complete cores} for $\omega$-categorical structures is, as opposed to the finite case, non-trivial and was first  established in~\cite{Cores-journal}; they are again $\omega$-categorical and unique up to isomorphism. 

\begin{example}
    The model-complete core of 
    \begin{itemize}
        \item $(\mathbb Q,\leq )$ is a single point with equality;
        \item $(\bQ,<)$ is $(\bQ,<)$;
        \item the random graph is the countably infinite clique;
        \item $K_{\omega,\omega}$, the complete bipartite graph with two countable parts, is a single edge, $K_2$.
    \end{itemize}
\end{example}

The model-complete core $\bA'$ of a structure $\bA$  has desirable model-theoretic  properties, e.g., the orbits of its automorphism group on $n$-tuples are pp-definable in it.
As a result, it is systematically used in the study of constraint satisfaction in infinite domains~\cite{Book}. Even in second-generation complexity classifications, a standard ``preprocessing'' step when determining the complexity of the CSP of a template is replacing the template by its model-complete core, see e.g.,~\cite{MottetPinskerSmooth}. To do so, combinatorial effort is made to  \emph{compute} this core. Unfortunately, neither the original  existence proof from~\cite{Cores-journal}, nor the later alternative proof from~\cite{EquationsInOliog} helps with that task as they are non-constructive, relying on a weak version of the axiom of choice, the compactness theorem of first-order logic.

While not providing a computing procedure either as it relies on the compactness theorem again, the following more recent result due to Mottet and Pinsker has helped finding the model-complete core in practice. 

\begin{theorem}[Mottet and Pinsker~\cite{MottetPinskerCores}]\label{thm:MottetPinskerCores}
    Let $\bC$ be a  first-order reduct  of a  finitely bounded homogeneous Ramsey structure $\bA$. Then its model-complete core $\bC'$ is  a first-order reduct of a finitely bounded homogeneous Ramsey substructure $\bA' \subseteq \bA$. The relations of $\bC'$ are defined by the same quantifier free formulas in $\bA'$ as were those of  $\bC$ in $\bA$.
\end{theorem}

In this paper, we will provide an algorithm which outputs $\bA'$ and $\bC'$ in the situation of this theorem; this will be of help for our bi-interpretability questions.

\subsection{Our contributions}\label{sect:contributions}

\subsubsection{The complexity of bi-interpretability}
Our first main contribution lies in showing that the equivalence relation of pp-bi-inter\-pretability, underlying the  algebraic approach to infinite-domain CSPs since~\cite{Topo-Birk},  is reasonable both from the  perspective of computation as well as from the perspective of descriptive set theory. We start with the former, where we show the following variant of Theorem~\ref{thm:DecOfDef} from~\cite{BPT-decidability-of-definability}. 

\begin{restatable}{theorem}{DecOfInt}\label{thm:dec_of_int}
    The following problem is decidable:
    \begin{itemize}
    \item \textbf{Given:} Finitely bounded homogeneous Ramsey structures $\bA, \bB$ with first-order reducts $\bC,\bD$, respectively, whose   model-complete cores $\bC',\bD'$ are without algebraicity.
    \item \textbf{Question:} Are $\bC',\bD'$  (fo/ep)-bi-interpretable?
    \end{itemize}
    Moreover, if $\bA,\bB$ are transitive, then the same statement holds for pp-bi-interpretability.
\end{restatable}

Comparing this theorem with Theorem~\ref{thm:DecOfDef}, we observe the   appearance of model-complete cores, no algebraicity, and transitivity. With model-complete cores being the unique ``smallest'' template representing an $\omega$-categorical CSP, and our ability to compute the model-complete core (see Theorem~\ref{thm:mccore}), we are not disturbed by them. The assumption of no algebraicity in the model-complete core can be made without loss of generality by recent work in~\cite{removingAlgebraicity}; it is ubiquitous in the infinite-domain CSP literature (e.g., in the theory of smooth approximations~\cite{MottetPinskerSmooth}), and we refer to Section~\ref{sect:Prelims} for more details.  
Finally, the transitivity in the case of pp-bi-interpretability is the only true  restriction; it means that the automorphism group of the structure acts with a single orbit on its domain. 
A good example is, once again, the structure $(\mathbb Q,<)$: any element of $\bQ$ can be sent to any other element by an automorphism. Among all the complexity dichotomies only very few, e.g.,~the ones for unary structures~\cite{Bodirsky-Mottet} and  MMSNP~\cite{MMSNP_conf}, deal with non-transitive structures. Thus, this assumption still encompasses the vast majority of all work on Conjecture~\ref{conj:bp}.

Similarly to the computational perspective, we will show that the equivalence relation of bi-interpretability of $\omega$-categorical structures is reasonable from the point of view of descriptive set theory, under mild assumptions.

\begin{restatable}{theorem}{IntIsSmooth}\label{thm:biIntSmooth}
The equivalence relation of ep-bi-interpretability is a smooth equivalence relation on the class of  $\omega$-categorical model-complete cores without algebraicity. Moreover, pp-bi-interpretability is smooth on the subclass of structures that are additionally transitive. 
\end{restatable}

\subsubsection{Between interdefinability and bi-interpretability: bi-definability}

There is an intermediate equivalence relation on $\omega$-categorical structures between the one given by equality of polymorphism clones (characterized by pp-interdefinability) and the coarser one given by topological
isomorphism of polymorphism clones (characterized by pp-bi-interpretability): conjugacy of polymorphism clones. Namely, we can consider structures $\bA$ and $\bB$ equivalent if there is a bijection $\theta\colon \bA\to\bB$ between their domains such that the map $f(x_1,\ldots,x_n)\mapsto \theta\circ f(\theta^{-1}(x_1),\ldots,\theta^{-1}(x_n))$ is a bijection from the polymorphism clone of $\bA$ onto the polymorphism clone of $\bB$. In that case, it is easy to see that this map is a topological isomorphism between the two polymorphism clones, and hence $\bA$ and $\bB$  are pp-bi-interpretable. Structures whose polymorphism clones are conjugate in this sense are called \emph{pp-bi-definable}. If one replaces bi-interpretability by bi-definability in Theorems~\ref{thm:dec_of_int}, and~\ref{thm:biIntSmooth}, then one can drop the assumptions on no algebraicity and transitivity. 
In fact, we first show the following statements, the former also being a generalization of Theorem~\ref{thm:DecOfDef}.

\begin{restatable}{theorem}{biDefDec}\label{thm:biDefDec}
    The following problem is decidable: 
    \begin{itemize}
    \item \textbf{Given:} Finitely bounded homogeneous Ramsey structures $\bA$, $\bB$ and  first-order reducts $\bC$, $\bD$ in a finite language.
    \item \textbf{Question:} Are the model-complete cores of the structures $\bC$, $\bD$ (fo/ep/pp)-bi-definable?
    \end{itemize}
\end{restatable}

\begin{restatable}{theorem}{biDefSmooth}\label{thm:biDefSmooth}
    The equivalence relation of being (fo/ep/pp)-bi-definable on $\omega$-categorical structures is smooth. 
\end{restatable}

We then use model-theoretic tricks going back to Rubin~\cite{Rubin} and Behrisch and Vargas-Garc\'{i}a~\cite{BehrischGarcia} to lift the result from bi-definability to bi-interpretability, or in other words from conjugacy to topological isomorphism.

\subsubsection{Computing the model-complete core} 
Regarding model-comp\-lete cores, we give the first constructive proof of their existence for first-order reducts of finitely bounded homogeneous Ramsey structures; in fact, we show that  in that situation, we can compute the model-complete core algorithmically. This is one of the ingredients of our proof of Theorem~\ref{thm:dec_of_int}. 

\begin{theorem}\label{thm:mccore}
   Given a first-order reduct $\bC$ of a finitely bounded homogeneous Ramsey structure $\bA$ we can compute a finitely bounded homogeneous Ramsey structure $\bA'$ and a first-order reduct $\bC'$ that is the model-complete core of $\bC$.  
\end{theorem}

The role of this theorem in the proof of Theorem~\ref{thm:biDefDec} is not only the computation of the model-complete core of a given structure. In fact, given $\bC$  we can compute an ``optimal''  finitely bounded homogeneous Ramsey structure $\bA'$ that yields the core $\bC'$ of $\bC$. We call the structure $\bA'$ the \emph{optimal presentation} of $\bC'$, and show that it is unique up to fo-bi-definability (Lemma~\ref{lem:optimalpresentation}). We believe that this insight is of independent interest in the quest to prove Conjecture~\ref{conj:bp}.

\subsection{Discussion}
We show decidability of bi-definability for model-complete cores of structures under the assumptions of Conjecture~\ref{conj:bp} 
while additionally assuming no algebraicity, which is no loss of generality, and the Ramsey property, which is a standard assumption and might be no loss of generality. Although not directly relevant for the study of the conjecture, it would also be interesting from a purely mathematical point of view whether 
bi-definability can be decided for the original structures rather than their model-complete cores: as will become apparent, our proof relies crucially on going to the model-complete core first. 

The same can be asked for the case of 
bi-interpretability: can we decide whether two given structures satisfying the assumptions of Theorem~\ref{thm:dec_of_int} are bi-interpretable? Moreover, can we drop the transitivity assumption in the primitive positive case?

Regarding the descriptive viewpoint, we can ask similar questions: is bi-interpretability smooth among $\omega$-categorical structures without algebraicity? Currently we only know this in the first-order case, by the breakthrough result of Paolini and Nies~\cite{PaoliniNies}. Note the interesting contrast of this fact with the discussion after Theorem~\ref{thm:DecOfDef}, which leaves precisely the first-order case open.

We remark that it is folklore that bi-definability as well as bi-interpretability are decidable for finite structures. The least obvious case here is that of  decidability of bi-interpretability in the primitive positive case: it is a consequence of the proof of  Birkhoff's theorem~\cite{Birkhoff}
 providing a bound on the power required in an interpretation.

\subsection{Organization of the paper}

The paper is organised as follows. In Section~\ref{sect:Prelims} we introduce relevant definitions and notions. In Section~\ref{sect:mccores} we show that, in the correct setting, we can compute the model-complete core of an input structure. We use this computability of the model-complete core in Section~\ref{sect:BiDefinability} to show that the equivalence relation of being bi-definable is tractable. In Section~\ref{sect:BiInterpretability} we extend work of Rubin to derive the bi-interpretability results from the bi-definability results obtained in the previous section. In the course of this we rely on a positive answer to a question of Bodirsky and Junker which is established in Section~\ref{sect:QuestionBodJun}. 

\section{Preliminaries}\label{sect:Prelims}

\subsection{The topology of pointwise convergence}\label{subsect:topo_pw}

Let $\Omega, \Delta$ be countable sets. The set of functions $\Omega \to \Delta$ will be denoted by $\Delta^{\Omega}$; the \emph{topology of pointwise convergence} on $\Delta^\Omega$ is the product topology on $\Delta^{\Omega}$ if we equip $\Delta$ with the discrete topology. It can be characterized as follows: $f\in \Delta^{\Omega}$ lies in the closure of $F \subseteq \Delta^{\Omega}$ if and only if, for every finite set $S\subseteq \Omega$ there is some $f' \in F$ such that $f$ and $f'$ agree when restricted to $S$. The induced subspace topology on subsets of $\Delta^\Omega$ is referred to as the topology of pointwise convergence as well. Spaces of functions are tacitly assumed to be equipped with the topology of pointwise convergence.

\subsection{Actions of groups, monoids, and clones}

Let $\Omega$ be a countable set. A \emph{permutation group on $\Omega$} is a subgroup $G$ of the group of all permutations of $\Omega$. We also say that \emph{$G$ acts on $\Omega$} and write $G\curvearrowright \Omega$ in this situation. Given a finite tuple $y \in \Omega^{< \omega}$, the \emph{stabiliser of $y$}, $G_y$, is the subgroup of $G$ consisting of all permutations that fix $y$ pointwise. 
More generally, for a set $X\subseteq \Omega$, $G_X$ denotes the subgroup of $G$ consisting of all permutations that fix $X$ pointwise.
An element $x\in \Omega$ is \emph{algebraic over $y$} if the orbit of $x$ under $G_y\curvearrowright \Omega$ is finite. The action $G\curvearrowright \Omega$ is \emph{without algebraicity} if for all $x\in \Omega, y \in \Omega^{<\omega}$, $x$ is algebraic over $y$ if and only if the element $x$ is an entry of the tuple $y$. 

A \emph{transformation monoid on $\Omega$} is a submonoid $M$ of $\Omega^\Omega$. We also say that \emph{$M$ acts on $\Omega$} and write $M \curvearrowright \Omega$ in this situation. A \emph{function clone on $\Omega$} is a subset $C$ consisting of finitary operations on $\Omega$, i.e., functions of the form $\Omega^n \to \Omega$ for $n\geq 1$, that contains all projection maps, and is closed under composition in the sense that for all $f,g_1,\dots,g_k \in C$ with $f \colon \Omega^k \to \Omega$ (i.e., $f$ has \emph{arity $k$}) it holds
\begin{equation*}
    f \circ (g_1,\dots, g_k) \in C.
\end{equation*}
We also say that \emph{$C$ acts on $\Omega$} and write $C \curvearrowright \Omega$ in this situation. If $C,D$ are function clones on $\Omega,\Omega'$ a bijection $\phi\colon C \to D$ is a \emph{clone isomorphism} if it preserves arities and compositions, and for all $n\geq i \geq 1$ sends the $i$th projection of arity $n$ to the $i$th projection of arity $n$. 

Let $(G / M / C)\curvearrowright \Omega$ be a permutation group / transformation monoid / function clone. A set $U \subseteq \Omega^n$, for $n \geq 1$, is said to be 
\begin{itemize}
    \item \emph{$G$-invariant}, if for all $g \in G$ and $x \in U$ it holds  $gx \in U$, evaluated coordinate-wise;
    \item \emph{$M$-invariant}, if for all $e \in M$ and $x \in U$ it holds  $ex \in U$, evaluated coordinate-wise;
    \item \emph{$C$-invariant}, if for all $f\in C$, and all $x_1,\dots,x_k \in U$, where $k$ is the arity of $f$, it holds  $f(x_1,\dots, x_k) \in U$, evaluated coordinate-wise.
\end{itemize}
Observe that if $U\subseteq \Omega^n$, for $n\geq 1$, is a $(G / M / C)$-invariant subset and $\sim$ is a $(G / M /C)$-invariant equivalence relation on $U$, then $G / M / C$ naturally induces a permutation group / transformation monoid / function clone on the set  $U/_\sim$ of equivalence classes; we write $(G / M / C) \curvearrowright U/_\sim$ for the induced permutation group / transformation monoid / function clone  on the quotient $U/_\sim$.

Permutation groups and transformation monoids are naturally equipped with the topology of pointwise convergence. The topology of pointwise convergence on a function clone $C\curvearrowright \Omega$ refers to the sum topology of the spaces $C \cap \Omega^{\Omega^n}$, $n \geq 1$. These topologies turn them into topological groups / monoids / clones, in the sense that composition is continuous, and, for groups, also the inversion map is. 
We say that a permutation group on $\Omega$ is \emph{topologically closed} if it is closed in the space of all  permutations of $\Omega$, a transformation monoid on $\Omega$ is \emph{topologically closed} if it is closed in the space of unary operations of $\Omega$, and a function clone on $\Omega$ is \emph{topologically closed} if it is closed in the space of finitary operations on $\Omega$.

\begin{definition}
    An \textit{isomorphism of permutation group actions} $G\curvearrowright\Omega$ and $H\curvearrowright\Omega'$ is a bijection $\theta\colon\Omega\to\Omega'$ with the property that the assignment $\alpha\mapsto\alpha^\theta:=\theta \circ \alpha \circ \theta^{-1}$ is an isomorphism from $G$ onto $H$. We write $G^\theta$ for the conjugate action (by $\theta$) of $G$; then this says that $G^\theta=H$. Similarly, isomorphisms between actions of closed transformation monoids and closed function clones are defined. 
\end{definition}

We point out that this notion of isomorphism between permutation groups $G\curvearrowright \Omega$ and $H\curvearrowright \Omega'$ implies that $G,H$ are \emph{isomorphic as topological groups}, since the map $\alpha\mapsto\alpha^\theta$ is a group isomorphism $G\to H$ that is a homeomorphism of the underlying topological spaces; the analogous statements hold for transformation monoids and function clones.

\subsection{Structures and symmetries}

We assume structures to be countable and in a relational signature, and denote them by blackboard bold letters $\bA$. If $\bA$ is some structure and $R$ is some relational symbol of arity $n$ in the signature of $\bA$, we denote the interpretation of $R$ in $\bA$, some subset of its $n$-th power, by $R^{\bA}$. Abusing notation, we usually do not distinguish between the symbol $R$ and the interpretation $R^\bA$ when the structure is clear from the context.
If not specified otherwise the domain of a structure is assumed to $\Omega$, a countable set. 

Let $\bA, \bB$ be structures in the same signature with domains $\Omega,\Omega'$. A \emph{homomorphism} from $\bA$ to $\bB$ is a function $f\colon \Omega \to \Omega'$ that \emph{preserves} all relations, i.e., for every relational symbol $R$ it holds $f(R^\bA) \subseteq R^\bB$, or by abuse of notation $f(R) \subseteq R$. An \emph{isomorphism} $\bA\to\bB$ is an invertible homomorphism whose inverse is a homomorphism $\bB \to \bA$.
An \emph{endomorphism} of a $\bA$ is a homomorphism $\bA \to \bA$. An \emph{automorphism} of $\bA$ is an isomorphism $\bA\to \bA$. A \emph{polymorphism (of arity $n\geq 1$)} is a map $f\colon \Omega^n \to \Omega$, that \emph{preserves all relations of $\bA$} in the sense that for all $x_1,\dots,x_n \in R^\bA$ it holds $f(x_1,\dots,x_n)\in R^\bA$.

We will consider various sets of symmetries of a structure $\bA$:
\begin{itemize}
    \item by $\Aut(\bA)$ we denote the group of automorphisms of $\bA$, a permutation group on $\Omega$;
    
    \item by $\End(\bA)$ we denote the monoid of endomorphisms of $\bA$, a transformation monoid on $\Omega$;
    
    \item by $\Pol(\bA)$ we denote the set of polymorphisms of $\bA$, a function clone on $\Omega$.
\end{itemize}

It is not hard to see that $\Aut(\bA)$ / $\End(\bA)$ / $\Pol(\bA)$ are topologically closed, and that also the converse is true: every topologically closed permutation group / transformation monoid / function clone on $\Omega$ is the automorphism group/ endomorphism monoid / polymorphism clone of some structure on $\Omega$.

A structure $\bA$ is a \emph{model-complete core} if $\Aut(\bA)$ is dense in $\End(\bA)$, i.e., for every finite subset $S$ of the domain of $\bA$ and every endomorphism $e$, the restriction of $e$ to $S$ extends to an automorphism of $\bA$.

\subsection{Model theory}

A permutation group $G\curvearrowright\Omega$ is \emph{oligomorphic} if for all $n\geq 1$ the induced permutation group $G\curvearrowright \Omega^n$ has finitely many orbits. A structure $\bA$ is \emph{$\omega$-categorical} if $\Aut(\bA) \curvearrowright \Omega$ is oligomorphic.

A structure $\bA$ is \emph{homogeneous} if every isomorphism between finite (induced) substructures of $\bA$ extends to an automorphism of $\bA$. The \emph{age} of $\bA$ is the class of its finite induced substructures up to isomorphism. The structure $\bA$ is called \emph{finitely bounded} if there is a finite set $F$ of finite structures such that for any finite structure $\mathbb X$ we have: $\mathbb X$ belongs to the age of $\bA$ if and only if it contains no member of $F$ as an induced substructure. The structure $\bA$ has \emph{no algebraicity} if the action of its automorphism group $\Aut(\bA) \curvearrowright \Omega$ is without algebraicity.

A first-order (fo) formula is said to be \emph{existential positive~(ep)} if it is of the form $\exists x_1\cdots\exists x_n \phi$, where $\phi$ is quantifier free and without negations; it is said to be \emph{primitive positive~(pp)} if it is of the form $\exists x_1\cdots\exists x_n \phi$, where $\phi$ is a conjunction of atomic formulas. An $n$-ary relation $R$ on $\Omega$ (this simply means that $R\subseteq \Omega^n$) is \emph{(fo/ep/pp)-definable} in $\bA$ if there is some (fo/ep/pp)-formula $\phi$ in the signature of $\bA$ such that $a \in R$ if and only if $\phi(a)$ holds. 

In the following theorem, the first-order statement is a well-known consequence of the Theorem of Engeler, Ryll-Nardzewski and Svenonius (see, e.g.,~\cite{Hodges}), the existential positive statement is folklore, and the primitive positive statement is due to~\cite{BodirskyNesetrilJLC}.

\begin{theorem}
    Let $\bA$ be an $\omega$-categorical structure and let $R$ be a relation on $\Omega$.
    $R$ is (fo/ep/pp)-definable in $\bA$ if and only if it is preserved by all automorphisms / endomorphisms / polymorphisms of $\bA$.
\end{theorem}

Two structures $\bA,\bB$ are said to be \emph{(fo/ep/pp)-interdefinable} if they (fo/ep/pp)-define the same sets of relations; for $\omega$-categorical structures, the above theorem immediately implies that this is equivalent to them having the same automorphism groups / endomorphism monoids / polymorphism clones. The first step of abstraction from this notion is the following. 

\begin{definition}
Let $\bA,\bB$ be structures with domains $\Omega,\Omega'$. A bijection $\theta\colon \Omega\to\Omega'$ is a \emph{(fo/ep/pp)-bi-definition} of $\bA$ and $\bB$ if $\theta$ and $\theta^{-1}$ map (fo/ep/pp)-definable sets onto (fo/ep/pp)-definable sets.
\end{definition}

\begin{lemma}\label{lem:biDefAreActionIsos}
    Two $\omega$-categorical structures $\bA,\bB$ are (fo/ep/pp)-bi-definable if and only if their automorphism groups / endomorphism monoids / polymorphism clones have isomorphic actions.
\end{lemma}
\begin{proof}
    We show this for ep-bi-definability and endomorphism monoids. Let $\Omega,\Omega'$ be the domains of $\bA,\bB$.
    
    Since $\bA,\bB$ are $\omega$-categorical, ep-definable relations of $\bA,\bB$ are precisely the relations that are invariant under the actions of their respective endomorphism monoids. Thus, an isomorphism $\theta$ of the actions $\End(\bA)\curvearrowright \Omega$ and $\End(\bB) \curvearrowright \Omega'$ is an ep-bi-definition of $\bA$ and $\bB$. 

    Conversely, let $\theta$ be an ep-bi-definition of $\bA$ and $\bB$. By the symmetry of the situation it suffices to show that for all $e \in \End(\bA)$, it holds $\theta e \theta^{-1} \in \End(\bB)$, i.e., that it preserves all relations of $\bB$. This is the case because every relation of $\bB$ is of the form $\theta(X)$, where $X$ is an ep-definable relation in $\bA$, and $e$ preserves ep-definable relations.
\end{proof}

Abstracting even further, one arrives at the notion of bi-inter\-pretations.
Let $\bA,\bB$ be structures with domains $\Omega,\Omega'$. A \emph{(fo/ep/pp)-interpretation of $\bB$ in $\bA$} is a partial surjective map $\Omega^n \to \Omega'$ such that preimages of relations definable by atomic formulas over $\bB$ are (fo/ep/pp)-definable in $\bA$. The interpretation $\Gamma$ is said to be a 
\emph{(fo/ep/pp)-bi-interpretation} if there is a (fo/ep/pp)-interpretation $\Delta$ of $\bA$ in $\bB$ such that the graphs of the naturally defined compositions $\Delta \circ \Gamma$ and  $\Gamma\circ\Delta$ are (fo/ep/pp)-definable in $\bA,\bB$.

It is not hard to see that (fo/ep/pp)-bi-interpretations of $\bA,\bB$ induce topological isomorphisms between the automorphism groups / endomorphism monoids / polymorphism clones of $\bA,\bB$ (see, e.g.,~\cite{Topo-Birk}). But also the converse is true, unconditionally even in two of the three cases. This is due to Coquand, see~\cite{AhlbrandtZiegler} in the fo-setting, due to Bodirsky and Junker~\cite{BodirskyJunker} in the ep-setting, and due to Bodirsky and Pinsker~\cite{Topo-Birk} in the pp-setting.

\begin{theorem}\label{thm:BiIntVsTopIso}
    Let $\bA,\bB$ be $\omega$-categorical structures.
    \begin{itemize}
        \item If $\bA,\bB$ have topologically isomorphic automorphism groups / polymorphism clones, then they are (fo/pp)-bi-interpretable;
        \item If $\bA,\bB$ do not contain constant endomorphisms, the analogous implication holds for endomorphism monoids and ep-bi-interpretability.
    \end{itemize}
\end{theorem}

\subsection{Ramsey theory}
Let $\bA,\bB$ be structures with domains $\Omega,\Omega'$, and let $f\colon \Omega\to\Omega'$ be any function. Then $f$ is called \emph{$(\bA,\bB)$-canonical} if the image of any $\Aut(\bA)$-orbit of a tuple in $\bA$ is contained in an $\Aut(\bB)$-orbit. This implies that $f$ induces a mapping from $\Aut(\bA)$-orbits to $\Aut(\bB)$-orbits; this mapping is called the \emph{behaviour} of $f$. If $\bA,\bB$ are homogeneous and in a finite signature, and $k$ is a common upper bound for the arities of their relations, then this behaviour is completely determined by its restriction to orbits of $k$-tuples; hence all potential behaviours can be enumerated in a finite list. If the structures are moreover finitely bounded, then one can even determine the behaviours that are  actually possible.

\begin{theorem}[Bodirsky, Pinsker, and Tsankov~\cite{BPT-decidability-of-definability}]\label{thm:computingBehaviour}
There is an algorithm which, given two finitely bounded homogeneous structures $\bA, \bB$, outputs all possible behaviours of canonical functions from $\bA$ to $\bB$ (the number of such behaviours is finite).
\end{theorem}

The importance of canonical functions lies in the fact that they can be computed between finitely bounded homogeneous structures; for \emph{Ramsey structures}, they are even common in a certain sense. Before making this precise we give, for the sake of completeness, the definition of a Ramsey structure. We would like to point out that in this paper we shall only make use of the forthcoming consequence of the Ramsey property. 
A countable structure $\bA$ is said to be \emph{Ramsey} if for all $k \geq 1$ and all finite substructures $\mathbb X,\mathbb Y \subseteq \bA$ there is a finite substructure $\mathbb Z \subseteq \bA$ such that for all maps
\begin{equation*}
    c \colon \{e\colon \mathbb X \to \mathbb Z \mid e \text{ is an embedding} \} \to \{1,\dots,k\}    
\end{equation*}
there is an embedding $f\colon \mathbb Y \to \mathbb Z$ such that $c$ is constant on the following set
\begin{equation*}
    \{ f \circ e \mid e\colon \mathbb X \to \mathbb Y \text{ is an embedding} \}.
\end{equation*}
 
\begin{theorem}[Bodirsky, Pinsker, and Tsankov~\cite{BPT-decidability-of-definability}]\label{thm:canonisation}
    Let $\bA$ be a Ramsey structure, let $\bB$ be $\omega$-categorical, and let $\Omega, \Omega'$ be their domains. Let $f\colon \Omega\to\Omega'$ be any function. Then the closure of the set $$\{\beta\circ f\circ \alpha \,\mid\, \alpha\in\Aut(\bA),\beta\in\Aut(\bB)\}$$ 
    in the space of all functions from $\Omega$ to $\Omega'$ (see Section~\ref{subsect:topo_pw}) contains an $(\bA,\bB)$-canonical function.
\end{theorem}

\subsection{Descriptive set theory}
    The following material can be found in any textbook on descriptive set theory, e.g.,~\cite{Kechris}.

    A Polish space is a topological space that is \emph{separable} (it contains a countable dense subset), and \emph{completely metrisable} (its topology is induced by a complete metric). 
    Recall that a complete metric is one where every Cauchy-sequence converges. One example of a Polish space is the space $\mathbb R$ of the real numbers with the topology induced by the Euclidean metric. Another example  is the space of all countably infinite 0-1-sequences: equipped with the topology of pointwise convergence from Section~\ref{subsect:topo_pw}, it is a Polish space, which is not homeomorphic to $\mathbb R$ as it is compact. However, from the point of view of their Borel structure, the two spaces are isomorphic: we shall introduce this structure next.

    Formally, a \emph{Borel space} is a pair $(X,\cB)$ where $X$ is a topological space and $\cB$ is the smallest \emph{$\sigma$-algebra}, i.e., the smallest family of subsets of $X$ that is closed under complements and countable intersections, that  contains all open subsets of $X$. Sets contained in $\cB$ are referred to as \emph{Borel sets}. For example, open sets are Borel by definition, closed sets are, and so are countable unions of closed sets. As the $\sigma$-algebra $\cB$ is completely determined by the topological space $X$ we simply refer to $X$ as a Borel space, without explicitly mentioning $\cB$; however, in this context we will be interested in \emph{Borel maps} rather than, say, continuous ones. 
    A map between two Borel spaces $f\colon X \to Y$ is \emph{Borel} if all preimages of Borel sets are Borel sets. The map $f$ is an \emph{isomorphism} of Borel spaces if it is invertible and its inverse is Borel again. Two Borel spaces are \emph{isomorphic} if there is an isomorphism from one onto the other.
    
    A \emph{standard Borel space} is a Borel space whose underlying topological space is a Polish space. Kuratowski's theorem states that the only standard Borel spaces up to isomorphism are 
    \begin{itemize}
        \item finite discrete spaces;
        \item the countable discrete space;
        \item the real numbers $\bR$.
    \end{itemize}
    This implies that, viewed as Borel spaces, $\bR$ and $\bR^n$ are isomorphic Borel spaces for all $n\geq 1$. This contrasts the perspective of topology, where $\bR^n$ and $\bR^m$ are isomorphic as topological spaces, i.e., homeomorphic, if and only if $n = m$. 
    This illustrates nicely why Borel isomorphism, in contrast to topological isomorphism (also called homeomorphism), is the right generalisation of computable isomorphism  when moving from countable to uncountable. In computability theory, the set $\mathbb N$ of natural numbers is the central object and the relevant morphisms are computable maps. These  make all finite powers of $\mathbb N$ isomorphic from the computability perspective: there is a computable bijection from ${\mathbb N}^n$ onto $\mathbb N$. For Borel-reducibility, $\mathbb R$ is the central object, and Borel-isomorphism makes all finite powers of it isomorphic.

    The class of all countable $\omega$-categorical structures can be encoded, for example, as 0-1-sequences, and forms an uncountable standard Borel space, see~\cite[II.16.C]{Kechris}, and also~\cite{NiesSchlichtTent}. By the above, this space is Borel-isomorphic to $\mathbb R$. Therefore,  equivalence relations on this class can be  compared with respect to Borel reductions, already formally defined in Section~\ref{sect_intro_descriptive}. Since, as the notion of tractability for classification problems on all $\omega$-categorical structures, it is central to this work, we repeat the notion of smoothness:
    
    \begin{definition}[Smoothness]
        An equivalence relation $E$ on a Borel space $X$ is \emph{smooth} if it Borel-reduces to equality on the space $\bR$.
    \end{definition}

\section{On the model-complete core}\label{sect:mccores}

\subsection{Characterization of bi-definitions}

In this section all structures have a fixed countable set $\Omega$ as their domain unless explicitly defined otherwise. 
We will show that if a structure is a first-order reduct of a finitely bounded homogeneous Ramsey structure, then we can compute its model-complete core along with a finitely bounded homogeneous Ramsey expansion of this model-complete core. 

\begin{lemma}\label{lem:criterion-bidefinability}
    Let $\bA,\bB$ be $\omega$-categorical homogeneous structures and let $\bC,\bD$ be first-order reducts in the same signature $\tau$. Then the following are equivalent:
    \begin{itemize}
        \item[(1)] $\bA$ and $\bB$ are fo-bi-definable by an isomorphism $\bC \to \bD$; 
        \item[(2)] There exist injective functions $f,g\colon \Omega\to \Omega$ such that
        \begin{itemize}
            \item[(i)] $f$ is $(\bA,\bB)$-canonical;
            \item[(ii)] $g$ is $(\bB,\bA)$-canonical;
            \item[(iii)] $g\circ f$ preserves all  orbits of $\bA$;
            \item[(iv)] $f\circ g$ preserves all  orbits of $\bB$. 
        \end{itemize}
    Moreover, for every relation $R\in \tau$ we require that
    \begin{itemize}
        \item[(v)] $f$ takes $\bA$-orbits in $R^\bC$ to $\bB$-orbits in $R^\bD$;
        \item[(vi)] $g$ takes $\bB$-orbits in $R^\bD$ to $\bA$-orbits in $R^\bC$.
    \end{itemize}
    \end{itemize}
\end{lemma}
\begin{proof}
    (2) implies (1): Set $\Omega':=f(\Omega)$.
    Let $\bB',\bD'$ be the restrictions of $\bB,\bD$ to $\Omega'$. Since $f \circ g$ 
    preserves the orbits of $\bB$ by~(iv), it is an embedding of $\bB$ into $\bB'$, and in particular the age of $\bB$ is contained in that of $\bB'$. The converse inclusion being trivial, we conclude that $\bB$ and $\bB'$ have the same age.   

    We now prove that $\bB'$ is homogeneous. To see this, note that  $G := \Aut(\bA)$ acts on $\Omega'$ by its conjugate action $G^f$. Observe that this action preserves $\bB$-orbits: by~(i) and~(ii) $f$ and $g$ induce functions between $\bA$-orbits and $\bB$-orbits, respectively, and~(iii) and~(iv) imply that these functions are bijections. Since any  $\alpha\in G$ does not move $\bA$-orbits, we see that $\alpha^f(a)$ belongs to the $\bB$-orbit of $a$ for any tuple $a$ in $\Omega'$. Thus, the action $G^f\curvearrowright \Omega'$ is an action by automorphisms of $\bB'$, i.e., $G^f\leq \Aut(\bB')$. 
    We next observe that this action is transitive on every $\bB$-orbit (restricted to $\Omega'$). Indeed, for arbitrary tuples $a,b$ of $\bB'$ in the same $\bB$-orbit, we have that $f^{-1}(a)$ and $f^{-1}(b)$ belong to the same $\bA$-orbit since $f$ is a bijection between $\bA$-orbits and $\bB$-orbits. Thus, there is $\alpha \in G$ sending $f^{-1}(a)$ to $f^{-1}(b)$; then $\alpha^f(a)=b$. Putting this together, if two tuples $a,b$ in $\bB'$ are isomorphic, then they belong to the same $\bB$-orbit, and hence there exists $\alpha^f\in G^f\leq \Aut(\bB')$ with $\alpha^f(a)=b$. This shows  homogeneity of $\bB'$, and actually also $G^f=\Aut(\bB')$ since $G^f$ is closed and dense in $\Aut(\bB')$. 
    
    In particular, this shows, using Lemma~\ref{lem:biDefAreActionIsos}, that $\bA$ and $\bB'$ are fo-bi-definable via $f$. Moreover, item~(v) and item~(vi) show that $f$ is an isomorphism $\bC \to \bD'$.
    By Fra\"{i}ss\'{e}'s theorem (see, e.g.~\cite{Hodges}), $\bB',\bB$ are isomorphic and any such isomorphism is an isomorphism $\bD' \to \bD$. Therefore, $\bA$ and $\bB$ are fo-bi-definable by an isomorphism $\bC \to \bD$.

    (1) implies (2): Let $f\colon\bC\to\bD$ be an isomorphism that is a fo-bi-definition of $\bA,\bB$. Note that $f$ takes $\bA$-definable sets that are minimal (with respect to inclusion) to $\bB$-definable sets that are minimal (with respect to inclusion). Thus it takes $\bA$-orbits to $\bB$-orbits, i.e., it is $(\bA,\bB)$-canonical. The same argument shows that $f^{-1}$ is $(\bB,\bA)$-canonical. Taking $g$ to be $f^{-1}$ it immediately follows that all properties of item~(2) are met.
\end{proof}

\subsection{Computing the model-complete core}

\begin{definition}
    Let $\bC$ be a first-order reduct of an $\omega$-categorical homogeneous structure $\bA$. We say that $\bC$ is \emph{optimally presented over $\bA$} if the range of any  endomorphism of $\bC$ intersects all orbits of $\bA$.
\end{definition}

The following lemma might sound surprising, but perhaps less so if one considers the fact that it does hold for finite structures replacing orbits by elements.

\begin{lemma}\label{lem:optimalpresentation}
    Let $\bC,\bD$ be structures in the same signature, optimally presented over $\omega$-categorical homogeneous Ramsey structures $\bA,\bB$ in a  finite language. Assume that $\bC,\bD$ are homomorphically equivalent. Then $\bA$ and $\bB$ are fo-bi-definable by an isomorphism $\bC \to \bD$.
\end{lemma}
\begin{proof}
    Let $f,g$ be homomorphisms from $\bC$ into $\bD$ and from $\bD$ into $\bC$, respectively. 
    As they are homogeneous in a finite language, the structures $\bA,\bB$ are $\omega$-categorical, and since they are additionally Ramsey we can assume that $f$ and $g$ are canonical with respect to $(\bA,\bB)$ and  $(\bB,\bA)$, respectively. Namely, by Theorem~\ref{thm:canonisation}, there is an $(\bA,\bB)$-canonical function $f'$ in the topological closure of $\{ \beta \circ f \circ \alpha \mid \alpha \in \Aut(\bA), \beta \in \Aut(\bB) \}$, which is easily seen to be a homomorphism $\bC \to \bD$, so we may replace $f$ by $f'$; similarly $g$ can be replaced by a$(\bB,\bA)$-canonical homomorphism $\bD \to \bC$. 
    Since $g\circ f$ and $f\circ g$ are endomorphisms of $\bC,\bD$, by their optimal presentation we conclude that $f$ and $g$ induce bijections between the $\bA$-orbits and the $\bB$-orbits respectively. Namely, for every $k\geq 1$, $g\circ f$ and $f \circ g$ induce surjective maps on the finite sets of $k$-orbits in $\bA$ and $\bB$.
    Moreover we can assume that $f\circ g$ acts like the identity on $\bB$-orbits (which also implies that  $g\circ f$ acts like the identity on $\bA$-orbits) as the following argument shows. The structure $\bB$ is homogeneous in a finite language so, choosing $k\geq 1$ to be larger than the arity of each of its relations, the behaviour of a $(\bB,\bB)$-canonical function $\Omega \to \Omega$ on orbits of $\bB$ is uniquely determined by its action on $k$-orbits.
    Thus, there is some $n\geq 1$ such that $(fg)^n$ acts like the identity on $k$-orbits of $\bB$, and replacing $f$ by $(f\circ g)^{n-1}\circ f$ we obtain the desired property.
    
    Now observe that $\bA,\bB$ together with their reducts $\bC,\bD$ and the functions $f,g$ satisfy all properties of item~(2) in Lemma~\ref{lem:criterion-bidefinability}. Hence, $\bA$ and $\bB$ are fo-bi-definable by an isomorphism $\bC\to\bD$.
\end{proof}

\begin{definition}
    Let $\bA$ be an $\omega$-categorical homogeneous structure and let $f \colon \Omega \to \Omega$ be an $(\bA,\bA)$-canonical function. It is \emph{range-rigid} if its behaviour is an idempotent function on orbits, i.e., if composing the behaviour of $f$ with itself yields a function that is equal to the behaviour of $f$.
\end{definition}

The following is a strengthened version of Theorem~\ref{thm:mccore}, which was announced in Section~\ref{sect:contributions}.

\begin{theorem}\label{thm:computeop}
    Let $\bA$ be a finitely bounded  homogeneous Ramsey structure and $\bC$ be a first-order reduct in a finite language. We can compute  a finitely bounded homogeneous Ramsey structure $\bA'$ and a first-order reduct $\bC'$ such that $\bC'$ is optimally presented over $\bA'$ and the model-complete core of $\bC$.
\end{theorem}
\begin{proof}
    We start by computing all $(\bA,\bA)$-canonical endomorphisms $f$ of $\bC$ that are range-rigid and hit an inclusion-minimal set of $\bA$-orbits. Observe that this is possible. 
    Namely, by Theorem~\ref{thm:computingBehaviour}, we can compute the finite set of behaviours of $(\bA,\bA)$-canonical functions, and for each behaviour we can then check if it preserves the relations of $\bC$, which are just unions of $\bA$-orbits, and if it is range-rigid and additionally its range intersects an inclusion minimal set of $\bA$-orbits.
    By Lemma~9 of~\cite{MottetPinskerCores}, the age of the range (viewed as a substructure of $\bA$) of any such $f$ is finitely bounded by structures of size which are bounded by the maximum of arities of relations of $\bA$ and its bounds. Moreover, it is a Fra\"{i}ss\'{e}-class, whose Fra\"{i}ss\'{e}-limit $\bA'$ has the Ramsey property. 
    For one of these structures $\bA'$, the reduct $\bC'$ that is defined by the same quantifier-free formulas that define $\bC$ in $\bA$ is the model-complete core of $\bC$.
    In fact, this is the case whenever $f$ belongs to a minimal closed left-ideal of $\End(\bC)$, a property which a priori seems non-trivial to check.

    To circumvent this, we now observe that any such $\bC'$ is optimally presented over $\bA'$: suppose not, and let $g\in \End(\bC')$ be a function to witness this. Applying Theorem~\ref{thm:canonisation} we may additionally assume that $g$ is $(\bA',\bA')$-canonical. 
    The age of $\bA'$ is induced by the range of some $f\in\End(\bC)$, and we may assume that this range is contained in $\bA'$ by universality of the latter. But then $g \circ f$ contradicts minimality of the set of $\bA$-orbits contained in the image of $f$.
    
    Note that any two structures $\bC'$ are homomorphically equivalent, as they are homomorphically equivalent to $\bC$. Now by  Lemma~\ref{lem:optimalpresentation}, any two of the Fra\"{i}ss\'{e}-limits $\bA'$ are fo-bi-definable, and the corresponding reducts isomorphic via a fo-bi-definition. This implies the statement we want to show: we then know that for any $\bA'$ obtained, the reduct $\bC'$ must be the model-complete core of $\bC$, since we know this is the case for one of them.
\end{proof}
 
\begin{remark}
The above computation of the model-complete core can be carried out in $2\mathrm{NEXPTIME}$ with a single $\mathrm{coNP}$ query.
Suppose we are given a finitely bounded homogeneous Ramsey structure $\mathbb{A}$ specified by its relations and its bounds, and a first-order reduct $\mathbb{C}$ specified by quantifier-free formulas over $\mathbb{A}$. Let $n\geq 1$ be larger than the size of the bounds of $\mathbb{A}$
as well as the arities of the relations of $\mathbb{A}$ and of $\mathbb{C}$. We compute in $2\mathrm{EXPTIME}$ the doubly exponentially large set $O$ of
$k$-orbits of $\mathbb{A}$ for $k\leq n$, together with the restriction maps between them and the representation of the relations of $\mathbb{C}$ as unions of
orbits. We then wish to obtain a function $f\colon O\to O$ such that
\begin{enumerate}
    \item it is compatible with restrictions to sub-orbits;
    \item it is range-rigid, i.e., it satisfies $f(f(x))=f(x)$ for all $x\in O$;
    \item it preserves the relations of $\mathbb{C}$;
    \item it has inclusion-minimal range among the functions satisfying (1)--(3).
\end{enumerate}
Conditions (1)--(3) can be checked in time polynomial in $|O|$ but for condition (4) we require the $\mathrm{coNP}$-oracle. Namely, consider the language
\[
    L:=\bigl\{(O,\mathbb{C},f) \mid \text{there is } g\colon O\to O
    \text{ satisfying (1)--(3) with }
    \operatorname{range}(g)\subsetneq\operatorname{range}(f)\bigr\},
\]
where the triples are encoded as above. 
Since, with respect to the size of the input of $L$, the witness $g$ is of polynomial size and is verifiable in polynomial time, $L$ is in $\mathrm{NP}$. 
Note that $(O,\mathbb{C},f)\notin L$ precisely says that $f$ satisfies (4), so we obtain a desired function as follows.
In $2\mathrm{NEXPTIME}$ we find $f\colon O\to O$ satisfying (1)--(3) and then query the $\mathrm{coNP}$-oracle to check that it meets (4).
 
Such a function $f$ is induced by an $(\mathbb{A},\mathbb{A})$-canonical endomorphism $\tilde{f}$ of $\mathbb{C}$ that is range-rigid and has an inclusion-minimal set of $\mathbb{A}$-orbits as its range. The proof of Theorem~\ref{thm:computeop} establishes that the range of $\tilde{f}$ specifies the age of a finitely bounded homogeneous Ramsey structure $\mathbb{A}'$ whose first-order reduct $\mathbb{C}'$ is the model-complete core of $\mathbb{C}$. 
The set of bounds of $\mathbb{A}'$ equals the set of bounds of $\mathbb{A}$ together with $O\setminus f(O)$, and the formulas that specify the relations of $\mathbb{C}'$ are the same formulas that specify the relations of $\mathbb{C}$.
\end{remark}

The following Lemma is not necessary but it shows that in this setting the existence of the model-complete core is provable without the axiom of choice.

\begin{lemma}\label{lem:coresZF}
    Let $\bC$ be a first-order reduct of a finitely bounded  homogeneous Ramsey structure $\bA$. Then $\bC$ has a model-complete core. 
\end{lemma}
\begin{proof}
    There exists an $(\bA,\bA)$-canonical endomorphism $f$ of $\bC$ which is range-rigid with respect to $\Aut(\bA)$ and which hits a set of $\bA$-orbits which is minimal with respect to inclusion among sets of $\bA$-orbits that are obtained from such functions; such a function can even be computed by finite boundedness, see Theorem~\ref{thm:computingBehaviour}. Let $\bA'$ be the Fra\"{i}ss\'{e}-limit of the age of 
    the structure induced by the range of $f$ in $\bA$; it exists by Lemma~6 of~\cite{MottetPinskerCores}. 
    
    We claim that the corresponding fo-reduct $\bC'$ has the property that for any endomorphism $g$ of $\bC'$ and any tuple $t$ in $\bC'$ there exists an endomorphism $h$ such that $h \circ g(t)=t$. Indeed, suppose otherwise. Then we can inductively (and by using topological closure) build an endomorphism $g'$ that has the same property for all tuples $t'$ in the $\bA'$-orbit of $t$. 
    Hence, there exists also an $(\bA',\bA')$-canonical range-rigid  such endomorphism $g''$. But then the endomorphism $g'' \circ f$ contradicts the minimality of the range of $f$.
       
    Hence, $\bC'$ is a model-complete core, and clearly it is homomorphically equivalent to $\bC$.
\end{proof}

We note that in the proof of Lemma~\ref{lem:coresZF}, we do not need range-rigidity, but we do to have a clean citation of~\cite{MottetPinskerCores}.

\section{On bi-definability}\label{sect:BiDefinability}

Again, we assume all structures to have a fixed countable set $\Omega$ as their domain unless explicitly defined otherwise.

We will show that the relation of being bi-definable is simple in two interpretations of the word. First, we will show that it is simple in the computational sense, i.e., that it is decidable if one restricts to a subclass of structures where asking such question makes sense, namely when the input structures can be presented in a finite way. Second, we will show that it is simple in the descriptive set theory sense, namely that it is smooth, i.e., as simple as possible among equivalence relations on Borel spaces.

\subsection{The computational complexity perspective}

\biDefDec*
\begin{proof}
    Let $n$ be the maximal arity of relations of $\bC,\bD$. Let $\bC',\bD'$ be the model-complete cores of $\bC,\bD$; by Theorem~\ref{thm:computeop} we can compute them along with finitely bounded homogeneous Ramsey structures $\bA',\bB'$ over which they are optimally presented. Let $\bC'',\bD''$ be the expansions of $\bC',\bD'$ by all xy-definable relations of arity at most $n$.

    We briefly explain why $\bC'',\bD''$ are computable from $\bC',\bD'$. The latter structures are model-complete cores so any fo-definable relation is in fact ep-definable. And in this context the definability of relations by (ep/pp)-formulas is decidable. This is the content of the main theorems of~\cite{BPT-decidability-of-definability}, Theorem~1, and Theorem~2. Thus, we may add the finitely many relations to $\bC',\bD'$ in a computable way.
    
    Suppose that $\bC',\bD'$ are xy-bi-definable via a bijection $\theta\colon\Omega\to\Omega$.
    This map is an xy-bi-definition of $\bC'',\bD''$ as well, and using $\theta$, we obtain a common signature $\tau$ of $\bC'',\bD''$ turning $\theta$ into an isomorphism $\bC'' \to \bD''$; in particular they are homomorphically equivalent. Note that, $\bC', \bC''$ as well as $\bD',\bD''$ have the same sets of endomorphisms so $\bC'',\bD''$ are optimally presented over $\bA',\bB'$.
    Thus, Lemma~\ref{lem:optimalpresentation} implies that $\bA'$ and $\bB'$ are fo-bi-definable by an isomorphism $\bC'' \to\bD''$.
    Applying Lemma~\ref{lem:criterion-bidefinability} 
    provides us with functions $f,g\colon \Omega \to \Omega$ such that
    \begin{itemize}
        \item[(i)] $f$ is $(\bA',\bB')$-canonical;
        \item[(ii)] $g$ is $(\bB',\bA')$-canonical;
        \item[(iii)] $g \circ f$ preserves all orbits of $\bA'$;
        \item[(iv)] $f \circ g$ preserves all orbits of $\bB'$.
    \end{itemize}  
    Moreover, for every relation $R$ in the common signature of $\bC'',\bD''$
    \begin{itemize}
        \item[(v)] $f$ takes $\bA'$-orbits in $R^{\bC''}$ to $\bB'$-orbits in $R^{\bD''}$;
        \item[(vi)] $g$ takes $\bB'$-orbits in $R^{\bD''}$ to $\bA'$-orbits in $R^{\bC''}$.
    \end{itemize}    
    Conversely, if such functions exist, then $\bC'',\bD''$ are isomorphic by a fo-bi-definition of $\bA',\bB'$. 
    
    Hence, to correctly decide xy-bi-definability of $\bC',\bD'$, it is sufficient to check if there exists a common signature $\tau$ for $\bC'',\bD''$ as well as functions $f,g$ as in item~(2) of Lemma~\ref{lem:criterion-bidefinability}. 
    This poses no problem as there are only finitely many ways to assign a signature to $\bC'',\bD''$ and $\bA',\bB'$ are finitely bounded homogeneous, so we can compute the finitely many behaviours of canonical functions by Theorem~\ref{thm:computingBehaviour}.
\end{proof}

\subsection{The descriptive set theory perspective}

This section contains an elementary proof of the fact that for $\cS \in \{\Aut,\End,\Pol\}$ the assignment $\bA \mapsto \Th(\bE_{\cS(\bA)})$ is a Borel reduction from the equivalence relation of bi-definability on the class of $\omega$-categorical structures to equality on the standard Borel space $\bR$. Here, $\Th(-)$ denotes the first-order theory of a structure.

\begin{definition}
    Let $\mathcal{A} \subseteq \bigcup_{n \geq 1} \Omega^{\Omega^n}$ and define $\cA^{(n)} := \cA \cap \Omega^{\Omega^n}$, $n \geq 1$.
    We define the structure $\bE_\cA$ to have domain $\Omega$, and for all $k,n\geq 1$ to have the relation 
    \begin{equation*}
            O_{k,n}^{\bE_\cA} = \{ (x_1,\dots, x_k,f(x_1,\dots,x_k) ) \, \mid \, x_1,\dots,x_k \in \Omega^n, \, f\in \cA^{(k)} \} \subseteq \Omega^{(k+1)n}.
    \end{equation*}

\end{definition}

\begin{lemma}\label{lem:OrbitalStructuresIsomorphicIff}
    Let $\cA\subseteq \bigcup_{n \geq 1} \Omega^{\Omega^n}$ and $\cB\subseteq \bigcup_{n \geq 1} \Omega^{\Omega^n}$ be topologically closed. Then the isomorphisms $\theta\colon \bE_\cA \to \bE_\cB$ are precisely the bijections $\Omega \to \Omega$ such that $\cB = \theta \circ \cA \circ \theta^{-1}$.
\end{lemma}
\begin{proof}
    First, let $\theta\colon \Omega \to \Omega$ be a bijection such that
    \begin{equation*}
        \cB = \theta \circ \cA \circ \theta^{-1}.
    \end{equation*} 
    To show that it is an isomorphism of structures $\bE_\cA \to \bE_\cB$ it suffices to show that $\theta( O_{k,n}^{\bE_\cA} ) \subseteq O_{k,n}^{\bE_\cB}$ holds for all $k,n \geq 1$. But this is immediate from the definition of the relation $O_{k,n}$. 

    Conversely let $\theta$ be an isomorphism $\bE_\cA \to \bE_\cB$. It suffices to prove that $\theta \circ \cA \circ \theta^{-1} \subseteq \cB$.  Let $k\geq 1$ and $f\in \cA^{(k)}$; we show that $\theta \circ f \circ \theta^{-1} \in \cB$. Since $\cB$ is closed with respect to the topology of pointwise convergence we may equivalently show that for all $n \geq 1$ and all $x_1,\dots,x_k \in \Omega^n$ there is $g \in \cB^{(k)}$ with $\theta \circ  f \circ\theta^{-1} (x_1,\dots,x_k) = g(x_1,\dots,x_k)$. But this follows from the fact that the tuple 
    \begin{equation*}
        \big (\theta^{-1}(x_1),\dots,\theta^{-1}(x_k),f(\theta^{-1}(x_1),\dots,\theta^{-1}(x_k))\big)
    \end{equation*}
    is contained in $O_{k,n}^{\bE_\cA}$, so an application of the isomorphism $\theta$ shows 
    \begin{equation*}
        (x_1,\dots,x_k,\theta \circ f\circ \theta^{-1}(x_1,\dots,x_k)) \in O_{k,n}^{\bE_\cB},
    \end{equation*}
implying the existence of $g \in \cB$ with $\theta \circ f\circ \theta^{-1}(x_1,\dots,x_k) = g(x_1,\dots,x_k)$.
\end{proof}

\begin{corollary}\label{cor:orbitalIsomorphicIff}
    Let $\bA,\bB$ be structures and let $\cS\in\{\Aut,\End,\Pol\}$. Then the isomorphisms of structures $\bE_{\cS(\bA)}\to\bE_{\cS(\bB)}$ correspond to isomorphism of the actions $\cS(\bA)\curvearrowright \Omega \to \cS(\bB) \curvearrowright \Omega$.
\end{corollary}
\begin{proof}
    For a structure $\bA$ the sets of operations $\Aut(\bA)$, $\End(\bA)$, and $\Pol(\bA)$ are topologically closed. Thus Lemma~\ref{lem:OrbitalStructuresIsomorphicIff} applies.
\end{proof}

\begin{lemma}\label{lem:orbitStructuresOmegaCat}
    If $\bA$ is an $\omega$-categorical structure, then the structures $\bE_{\Aut(\bA)},\bE_{\End(\bA)}, \bE_{\Pol(\bA)}$ are $\omega$-categorical themselves.
\end{lemma}
\begin{proof}
    Since $\omega$-categoricity translates to reducts we simply need to observe that $\Aut(\bA)$ is contained in the automorphism groups of $\bE_{\Aut(\bA)},\bE_{\End(\bA)},\bE_{\Pol(\bA)}$ respectively. But this is clearly the case as automorphisms of $\bA$ preserve the relations of the structures $\bE_{\Aut(\bA)},\bE_{\End(\bA)},\bE_{\Pol(\bA)}$.
\end{proof}

We would like to point out that the first-order version of the following results essentially appears as Proposition~4.2 in~\cite{NiesSchlichtTent}.

\biDefSmooth*
\begin{proof}
    We show the statement for pp-bi-definability as the other two cases are proved analogously. Smoothness of the equivalence relation of pp-bi-definability is proved by Borel reducing it to equality on a standard Borel space. This reduction is given by the mapping $\bA \mapsto \Th(\bE_{\Pol(\bA)})$. We first demonstrate that it is in fact a reduction, i.e., that $\bA,\bB$ are pp-bi-definable if and only if $\Th(\bE_{\Pol(\bA)})=\Th(\bE_{\Pol(\bB)})$. To this end observe that Lemma~\ref{lem:biDefAreActionIsos} implies that $\bA,\bB$ are pp-bi-definable if and only if the actions of their function clones $\Pol(\bA)\curvearrowright\Omega$ and $\Pol(\bB) \curvearrowright \Omega$ are isomorphic. By Corollary~\ref{cor:orbitalIsomorphicIff} this is equivalent to the structures $\bE_{\Pol(\bA)}$ and $\bE_{\Pol(\bB)}$ being isomorphic, and, as these structures are $\omega$-categorical by Lemma~\ref{lem:orbitStructuresOmegaCat}, their theories are equal.

    It remains to show that the reduction is Borel, which is standard (see~\cite{NiesSchlichtTent}, and also~\cite{WEISmooth}). 
    Namely, we can assign in a Borel way to a structure $\bA$, the structure $\bE_{\Pol(\bA)}$. Moreover assigning the theory to a countable structure in a countable signature is Borel as well. The theory can be encoded as sequence in $\omega$, and $\omega^\omega$ is a standard Borel space isomorphic to the reals.
\end{proof}

\section{On bi-interpretability}\label{sect:BiInterpretability}

Again, we assume all structures to have a fixed countable set $\Omega$ as their domain unless explicitly defined otherwise.

\subsection{From bi-definability to bi-interpretability: the group case}

Given an $\omega$-categorical structure $\bA$ we define a new structure $\bA^\eq$ as follows: its domain $A^\eq$ consists of all equivalence classes of $\Aut(\bA)$-invariant equivalence relations on subsets of finite powers of $\Omega$; its relations are all $\Aut(\bA)$-invariant relations on this domain in the natural action of $\Aut(\bA)$ thereon. We identify the domain $\Omega$ of $\bA$ with a subset of $\bA^\eq$ in the natural way.

Note that every $\alpha \in \Aut(\bA)$ induces an automorphism of $\bA^\eq$ and that every automorphism is of that form, in particular $\Aut(\bA)\curvearrowright A^\eq$ and $\Aut(\bA^\eq)$ are equal. 

\begin{definition}[Density]
Let $G\curvearrowright X$ be a permutation group. A subset $Y\subseteq X$ is \emph{dense} if for all $x\in X$ there is finite $y\in Y^{<\omega}$ such that $G_y\subseteq G_x$. A subset of a structure is dense if it satisfies the definition with respect to the automorphism group of that structure.
\end{definition}

\begin{lemma}\label{lem:densityOfBiint}
    Let $\bA,\bB$ be $\omega$-categorical structures and $\Gamma$ be a fo-bi-interpretation of $\bB$ in $\bA$ with domain $U$ and kernel ${\sim}\subseteq U^2$. Then the quotient $U/_\sim$ is a dense subset of $\bA^\eq$.
\end{lemma}
\begin{proof}
    Let $\Omega'$ be the domain of $\bB$ and let $\tilde\Gamma \colon U/_\sim \to \Omega'$ be the bijection induced by $\Gamma$. 
    The fo-bi-interpretation $\Gamma$ induces a topological isomorphism $\phi\colon \Aut(\bA) \to \Aut(\bB)$,
    which is the unique map such that for all $\alpha \in \Aut(\bA)$ the following diagram commutes (see~\cite{AhlbrandtZiegler}). 

    \[\begin{tikzcd}
	   {U/_\sim} && \Omega' \\
	   \\
    	{U/_\sim} && \Omega'
	   \arrow["{\tilde\Gamma}", from=1-1, to=1-3]
        \arrow["\alpha"', from=1-1, to=3-1]
	   \arrow["{\phi(\alpha)}", from=1-3, to=3-3]
	   \arrow["{\tilde\Gamma}", from=3-1, to=3-3]
    \end{tikzcd}\]
    
    This precisely means that $\tilde\Gamma$ induces an isomorphism of permutation group actions $\psi\colon\Aut(\bA)\curvearrowright U/_\sim \to \Aut(\bB)\curvearrowright \Omega'$. 
    Thus the natural map $\Aut(\bA) \to \Aut(\bA)\curvearrowright U/_\sim$, which can be written as the composition $\psi^{-1} \circ \phi$, is a topological group isomorphism, and in particular open, which readily implies that $U/_\sim$ is dense in $\bA^\eq$.
\end{proof}

\begin{definition}[Entanglement]
    Let $\bA$ be an $\omega$-categorical structure. Let $B\subseteq \bA^\eq$. 
    Let $a\in \Omega$, and let $b/_\sim \in \bA^\eq$. We say that $a$ \emph{entangles} $b/_\sim$ if $a$ is contained in all tuples of $b/_\sim$. 
    We say that $a$ \emph{strongly entangles}  $b/_\sim\in B$ if it entangles $b/_\sim$ but no other element of $B$ (this depends on $B$ which will be clear from context). 
\end{definition}

\begin{lemma}\label{lem:entanglement_establishing}
Let $\bA$ be an $\omega$-categorical structure without algebraicity and let $B \subseteq \bA^\eq$ be a dense $\Aut(\bA)$-invariant set such that $\Aut(\bA)\curvearrowright B$ is without algebraicity. The following statements hold:
\begin{itemize}
    \item[(i)] $a\in \Omega$ entangles $b/_\sim\in B$ if and only if $a$ is algebraic over $b/_\sim$.
    \item[(ii)] every $a\in \Omega$ entangles some $b/_\sim\in B$.
    \item[(iii)] every $b/_\sim\in B$ is strongly entangled by some $a\in \Omega$.
    \item[(iv)] if $a$ entangles $b/_\sim$, and $a'$ strongly entangles $b/_\sim$ then $a$ and $a'$ are equal. 
    \item[(v)] every $a$ strongly entangles some $b/_\sim$.
\end{itemize}
\end{lemma}
\begin{proof}
    We write $G:=\Aut(\bA^\eq)$.

    (i)~For the forward direction, note that the set of all elements of $\Omega$ that entangle $b/_\sim$  is finite, invariant under  $G_{b/_\sim}$, and contains $a$. For the backward implication note that, if $a$ is algebraic over $b/_\sim$, then $a$ is algebraic over every $c\in b/_\sim$, so $a$ has to be contained in $c$ since $\bA$ has no algebraicity.
    
    (ii)~By the density of $B$ there exist $b_1/_{\sim_1},\ldots,b_n/_{\sim_n}\in B$ such that $G_a\supseteq G_{b_1/_{\sim_1},\ldots,b_n/_{\sim_n}}$. 
    Note that if we pick any tuples $c_i$ in $b_i/_{\sim_i}$, then $G_a\supseteq G_{c_1,\ldots,c_n}$, in particular $a$ is algebraic over the entries of $c_1,\dots,c_n$, which is a finite subset of $\Omega$. Thus, if we could pick all $c_i$ so that they do not contain $a$, we would arrive at a contradiction to $\bA$ not having algebraicity. Hence, there is $i$ such that $a$ entangles $b_i/_{\sim_i}$.

    (iii)~Pick $(a_1,\ldots,a_k)\in b/_\sim$. If none of the $a_i$ strongly entangled $b/_\sim$, then for all $i$ there would be $b_i/_{\sim_i}$ distinct from $b/_\sim$ that is entangled by $a_i$. But then $(a_1,\dots,a_k)$, and in particular $b/_\sim$ would be algebraic over $b_1/_{\sim_1},\dots,b_k/_{\sim_k}$ by~(i), contradicting that $\Aut(\bA)\curvearrowright B$ is without algebraicity.

    (iv)~If $\alpha \in \Aut(\bA)$ fixes $a'$, then $\alpha(b/_\sim) = \alpha(b)/_\sim$ is entangled by $a'$. Since $a'$ strongly entangles $b/_\sim$, the automorphism $\alpha$ fixes $b/_\sim$. By item~(i), $a$ is algebraic over $b/_\sim$, and so the previous argument shows that it is also algebraic over $a'$. Since $\bA$ is without algebraicity we obtain $a=a'$.

    (v)~The element $a$ entangles some $b/_\sim$ by~(ii); by~(iii), we know that $b/_\sim$ is strongly entangled by some $a'$. By~(iv) it follows that $a=a'$.
\end{proof}
\begin{lemma}\label{lem:entanglementBijection}
    Let $\bA$ and $B\subseteq \bA^{\eq}$ be as in Lemma~\ref{lem:entanglement_establishing}. Then 
    entanglement is a bijection $\theta\colon \Omega\to B$ preserved by $\Aut(\bA)$. Hence, the actions of $\Aut(\bA)$ on $\Omega$ and $B$ are isomorphic via $\theta$. 
\end{lemma}
\begin{proof}
    The notions of entanglement and strong entanglement agree by Lemma~\ref{lem:entanglement_establishing}: if $a\in \Omega$ entangles $b/_\sim \in B$, then there is $a'\in \Omega$ that strongly entangles $b/_\sim$, so $a=a'$. Thus entanglement defines a mapping $\Omega\to B$ that is easily seen to be a bijection using Lemma~\ref{lem:entanglement_establishing}. It is easy to see that it is preserved by $\Aut(\bA)$.
\end{proof}

\begin{corollary}\label{cor:iso_setinduced}
    Let $\bA,\bB$ be $\omega$-categorical and without algebraicity. Then any topological isomorphism between $\Aut(\bA)$ and $\Aut(\bB)$ is an isomorphism of permutation group actions.
\end{corollary}
\begin{proof}
    Let $\phi\colon\Aut(\bA) \to \Aut(\bB)$ be a topological group isomorphism. It is induced by a fo-bi-interpretation $\Gamma$ of $\bB$ in $\bA$ (Coquand; see~\cite{AhlbrandtZiegler}). Let $U$ be the domain and $\sim$ be the kernel of $\Gamma$, and let $\tilde \Gamma \colon U/_\sim \to \Omega'$ be the induced bijection, where $\Omega'$ denotes the domain of $\bB$. Saying that $\phi$ is induced by $\Gamma$ simply means that it admits the following factorisation
\[\begin{tikzcd}
	{\Aut(\bA)} && {\Aut(\bA)\curvearrowright U/_\sim} \\
	\\
	&& {\Aut(\bB)}
	\arrow["{\psi_1}", from=1-1, to=1-3]
	\arrow["\phi"', from=1-1, to=3-3]
	\arrow["{\psi_2}", from=1-3, to=3-3]
\end{tikzcd}\]
    where $\psi_1$ is the canonical homomorphism, sending $\alpha \in \Aut(\bA)$ to its action on $U/_\sim$, and $\psi_2$ is the homomorphism of permutation groups induced by $\tilde\Gamma$. 
    
    The set $U/_\sim$ is clearly $\Aut(\bA)$-invariant, and it is moreover dense in $\bA^\eq$ by Lemma~\ref{lem:densityOfBiint}. Moreover, as $\psi_2$ is an isomorphism of permutation group actions (its surjectivity follows from the surjectivity of $\phi$), the action $\Aut(\bA)\curvearrowright U/_\sim$ is without algebraicity. An application of Lemma~\ref{lem:entanglementBijection} shows that $\psi_1$ is an isomorphism of permutation group actions. Thus $\phi$ is an isomorphism of permutation group actions as well.   
\end{proof}

\begin{corollary}[Nies and Paolini~\cite{PaoliniNies}]
    Fo-bi-interpretability between $\omega$-categorical structures without algebraicity is smooth.
\end{corollary}
\begin{proof}
    Two $\omega$-categorical structures $\bA,\bB$ without algebraicity are bi-interpretable if and only if their automorphism groups are topologically isomorphic (Coquand; see~\cite{AhlbrandtZiegler}); this is the case if and only if the automorphism groups have isomorphic actions, by Corollary~\ref{cor:iso_setinduced}; this is the case if and only if $\bA$ and $\bB$ are fo-bi-definable (see Lemma~\ref{lem:biDefAreActionIsos}) and by Theorem~\ref{thm:biDefSmooth} this relation is smooth.
\end{proof}

\subsection{Extending to the endomorphism monoid}

In the following we show an analogue of Corollary~\ref{cor:iso_setinduced} for endomorphism monoids. We let $\bA,\bB$ be $\omega$-categorical structures and let $\Gamma$ be an ep-bi-interpretation of $\bB$ in $\bA$. Again, $U$ denotes the domain, and $\sim$ denotes the kernel of $\Gamma$. Recall that $U$ and $\sim$ are $\End(\bA)$-invariant so $\End(\bA)\curvearrowright U/_\sim$. 

\begin{lemma}\label{lem:EntanglementMonoids}
    Let $\bA$ be an $\omega$-categorical structure without algebraicity and let $U/_\sim \subseteq \bA^\eq$ be a dense set, such that $U$ and $\sim$ are $\End(\bA)$-invariant and moreover $\Aut(\bA)\curvearrowright U/_\sim$ is without algebraicity. Then the entanglement relation between $\Omega$ and $U/_\sim$ is preserved by all finite-to-one endomorphisms. 
\end{lemma}
\begin{proof}
    By applying Lemma~\ref{lem:entanglementBijection} we obtain that entanglement is a bijection $\Omega\to U/_\sim$ that is moreover preserved by $\Aut(\bA)$.
    
    Let $e\in\End(\bA)$ be finite-to-one, and assume that $a\in \Omega$ entangles $u/_\sim \in U/_\sim$; we want to show that $e(a)$ entangles $e(u)/_\sim$.    
    Suppose not, then some other element $y \in \Omega$ entangles $e(u)/_\sim$. 
    Let $X$ be the finite pre-image of $y$ under $e$. As $\bA$ has no algebraicity there exists $\alpha\in \Aut(\bA)$ that fixes $a$ such that no element of $X$ appears in $\alpha(u)$. Since $a$ also entangles $\alpha(u)_\sim$, and entanglement is a bijection we have $\alpha(u) \in u/_\sim$. But now $e(\alpha(u)) \in e(u)/_\sim$ does not contain $y$, a contradiction.
\end{proof}

\begin{corollary}\label{cor:entPreservedByEndos}
    Let $\bA, U, {\sim}$ be as in Lemma~\ref{lem:EntanglementMonoids} and additionally assume that the subset of finite-to-one endomorphisms is dense in $\End(\bA)$. Then entanglement is preserved by $\End(\bA)$. 
\end{corollary}

The proof of the following result is analogous to the proof of Corollary~\ref{cor:iso_setinduced} except for the fact that, unlike in the first-order case, it was previously not known whether a topological isomorphism between endomorphism monoids was always induced by an ep-bi-interpretation, see Theorem~\ref{thm:BiIntVsTopIso}. We will provide a positive answer under the assumption of no algebraicity in Section~\ref{sect:QuestionBodJun} to then obtain the following:

\begin{corollary}\label{cor:IsoIsSetInduced}
    Let $\bA$ be an $\omega$-categorical structure without algebraicity such that its finite-to-one endomorphisms are dense in $\End(\bA)$. Let $\bB$ be $\omega$-categorical without algebraicity. Then any topological isomorphism between $\End(\bA), \End(\bB)$ is an isomorphism of transformation monoid actions.
\end{corollary}
\begin{proof}
    Proposition~\ref{prop:isomEndoEpBiInterpretNoAlg} shows that any topological isomorphism $\phi\colon\End(\bA)\to\End(\bB)$ is induced by an ep-bi-interpretation between $\bA$ and $\bB$. Using the same notation as in the proof of Corollary~\ref{cor:iso_setinduced} we see that $\phi$ is the composition of $\psi_1\colon \End(\bA) \to \End(\bA)\curvearrowright U/_\sim$ and $\psi_2\colon \End(\bA) \curvearrowright U/_\sim \to \End(\bB)$. Corollary~\ref{cor:entPreservedByEndos} shows that $\psi_1$ is an isomorphism of transformation monoid actions and $\psi_2$ automatically is. Thus $\phi$ is one as well.
\end{proof}

\subsection{On bi-interpretability}

In this section we use the assumption of no algebraicity to upgrade our tractability results of bi-definability to a tractability result of bi-interpretability. For this we additionally need a result due to Behrisch and Vargas-Garc\'{i}a. Recall that a structure is transitive if its automorphism group acts transitively on its domain.

\begin{theorem}[{\cite[Theorem 4.3]{BehrischGarcia}}]\label{thm:liftingSetInducedIso}
    Let $\bA$ be a transitive structure, and let $\phi\colon \Pol(\bA) \to \Pol(\bB)$ be a surjective clone homomorphism, whose restriction to the unary part, i.e., $\phi^{(1)}\colon\End(\bA) \to \End(\bB)$ is an isomorphism of transformation monoid actions. Then $\phi$ is an isomorphism of function clone actions.    
\end{theorem}

\DecOfInt*
\begin{proof} 
    Let $\bC', \bD'$ be the model-complete cores of $\bC,\bD$. The notions of fo-definability and ep-definability agree for model-complete cores, so they are fo-bi-interpretable if and only if they are ep-bi-interpretable. 
    
    The proof in the existential positive setting can easily be adapted from the proof in the primitive positive setting so we proceed to prove the latter.
    By Theorem~\ref{thm:biDefDec} we can decide if $\bC'$ and $\bD'$ are pp-bi-definable, so it suffices to argue that, in this setting, this is equivalent to them being pp-bi-interpretable. More concretely we need to show that the pp-bi-interpretability of $\bC'$ and $\bD'$ implies that they are pp-bi-definable. 
    To this end 
    let the topological clone isomorphism $\phi\colon \Pol(\bC') \to \Pol(\bD')$ be a witness of the pp-bi-interpretability of 
    $\bC'$ and $\bD'$. Its restriction to unaries $\phi^{(1)}\colon \End(\bC') \to \End(\bD')$ is an isomorphism of transformation monoid actions by Corollary~\ref{cor:IsoIsSetInduced}; this use the fact that for model-complete cores the automorphisms are dense in the monoid of endomorphisms. Observe further that $\bC'$ is transitive: clearly $\bC$ is, being a first-order reduct of a transitive structure, and it is not hard to see that transitivity translates to the model-complete core as well. Thus, Theorem~\ref{thm:liftingSetInducedIso} in combination with Lemma~\ref{lem:biDefAreActionIsos} provides the desired conclusion.
\end{proof}

\IntIsSmooth*
\begin{proof}
    Smoothness of ep-bi-interpretability follows immediately from Theorem~\ref{thm:biDefSmooth}, Corollary~\ref{cor:IsoIsSetInduced}, and Lemma~\ref{lem:biDefAreActionIsos}. The smoothness of pp-bi-interpretability follows from Theorem~\ref{thm:biDefSmooth}, Corollary~\ref{cor:IsoIsSetInduced},  Theorem~\ref{thm:liftingSetInducedIso}, and Lemma~\ref{lem:biDefAreActionIsos}, similarly to the proof of Theorem~\ref{thm:dec_of_int}.
\end{proof}

\section{A question of Bodirsky and Junker}\label{sect:QuestionBodJun}

\begin{theorem}[\cite{BodirskyJunker}]\label{thm:BodJun}
    Two $\omega$-categorical structures without constant endomorphisms are ep-bi-interpretable if and only if their endomorphism monoids are topologically isomorphic.
\end{theorem}

While the authors of the above theorem show that any ep-bi-interpretation gives rise to a topological isomorphism of endomorphism monoids they point to some restriction regarding the converse implication. Namely, some restriction regarding constant endomorphisms is necessary as any structure interpretable in a structure having a constant endomorphism will have a constant endomorphism itself. Further, one easily constructs two $\omega$-categorical structures with topologically isomorphic endomorphism monoids such that only one of them  has a constant endomorphism. It is left open, however, if such a bi-interpretability result holds if both structures have a constant endomorphism. 
We observe a positive answer.

\begin{proposition}\label{prop:isomEndoEpBiInterpret}
    Let $\bA,\bB$ be $\omega$-categorical structures such that either both or neither have a constant endomorphism. Then $\bA$ and $\bB$ are existential positive bi-interpretable if and only if $\End(\bA)$, and $\End(\bB)$ are topologically isomorphic.
\end{proposition}
\begin{proof}
    In~\cite{BodirskyJunker} it is shown that ep-bi-inter\-pretations of $\omega$-cate\-gorical structures induce isomorphisms on the level of endomorphism monoids.

    Conversely, assume that $\End(\bA)$, and $\End(\bB)$ are topologically isomorphic and that $\bA$ has a constant endomorphism precisely if $\bB$ has one. By~\cite[Lemma 2.9]{MarimonPinskerReconstruction} $\bA$ and $\bB$ are existential positive bi-interpretable if and only if the topologically closed clones generated by $\End(\bA)$ and $\End(\bB)$ are topologically isomorphic. To see that this is the case let $\phi\colon \End(\bA) \to \End(\bB)$ be a topological isomorphism. It extends to a topological isomorphism of the generated clones if $\phi$ and $\phi^{-1}$ preserve constant functions, see~\cite[Proposition 5]{BPPReconstructingClone}. In case neither $\bA$ nor $\bB$ have constant endomorphisms this condition is vacuously true. But $\phi$ also provides a correspondence between the constant functions in case both, $\bA$ and $\bB$ have constant endomorphisms: in this case $e \in \End(\bA)$ is constant if and only if for all $f\in \End(\bA)$ it holds $e = e \circ f$; clearly this is true for constant endomorphisms and conversely, if $e$ has said property, then we may take a constant endomorphism $c\in \End(\bA)$ to see that $e = e \circ c$ is constant. The same being true for $\End(\bB)$ we conclude that $\phi$ extends to a topological isomorphism between the generated clones.
\end{proof}

We now further observe that in the case of no algebraicity, we can additionally drop the assumption on constant endomorphisms.

\begin{proposition}\label{prop:isomEndoEpBiInterpretNoAlg}
    Let $\bA,\bB$ be $\omega$-categorical and without algebraicity. Then any topological isomorphism between $\End(\bA),\End(\bB)$ maps constants to constants. Consequently, $\End(\bA)$ and $\End(\bB)$ are topologically isomorphic if and only if $\bA,\bB$ are ep-bi-interpretable. 
\end{proposition}
\begin{proof}
    The restriction of any topological isomorphism $\phi$ between $\End(\bA)$ and $\End(\bB)$ to the sets of invertibles $G:=\Aut(\bA)$ and $H:= \Aut(\bB)$ is an isomorphism of permutation group actions by Corollary~\ref{cor:iso_setinduced}. For $e$ in $\End(\bA)$, set $L(e):=\{\alpha \in G \, \mid \, \alpha \circ e =e \}$. We have $\phi(L(e))=L(\phi(e))$. Observe that $L(e)=G_{e(\Omega)}$. If $e$ is constant, then the stabiliser $G_{e(\Omega)}$ is not the intersection of two strictly larger stabilisers; if $e$ is not constant, then it is, by no algebraicity. But $\phi$ restricts to an isomorphism of permutation group actions $G\to H$, so it has to preserve this property. Hence it must send constants to constants.

    The second statement now follows from Proposition~\ref{prop:isomEndoEpBiInterpret}. 
\end{proof}

\bibliographystyle{plain}
\bibliography{references}

@article{Birkhoff,
author={Garrett Birkhoff},
title={On the structure of abstract algebras},
journal={Mathematical Proceedings of the Cambridge Philosophical Society},
volume=31,
number=4,
year=1935,
pages={433--454}
}

@INPROCEEDINGS{Pinsker-sheep,
author = {Pinsker, Michael},
booktitle = {2022 IEEE 52nd International Symposium on Multiple-Valued Logic (ISMVL)},
title = {Current Challenges in Infinite-Domain Constraint Satisfaction: Dilemmas of the Infinite Sheep},
year = {2022},
pages = {80-87},
doi = {10.1109/ISMVL52857.2022.00019},
publisher = {IEEE Computer Society},
address = {Los Alamitos, CA, USA},
month = May}

@article{GMSNP,
author = {Bienvenu, Meghyn and Cate, Balder Ten and Lutz, Carsten and Wolter, Frank},
title = {Ontology-Based Data Access: A Study through Disjunctive Datalog, CSP, and MMSNP},
year = {2015},
issue_date = {December 2014},
publisher = {Association for Computing Machinery},
address = {New York, NY, USA},
volume = {39},
number = {4},
issn = {0362-5915},
url = {https://doi.org/10.1145/2661643},
doi = {10.1145/2661643},
journal = {ACM Trans. Database Syst.},
articleno = {33},
numpages = {44}
}

@inproceedings{MMSNPFederVardi,
author = {Feder, Tom\'{a}s and Vardi, Moshe Y.},
title = {Monotone monadic SNP and constraint satisfaction},
year = {1993},
publisher = {Association for Computing Machinery},
address = {New York, NY, USA},
doi = {10.1145/167088.167245},
booktitle = {Proceedings of the Twenty-Fifth Annual ACM Symposium on Theory of Computing},
pages = {612–622},
numpages = {11},
location = {San Diego, California, USA},
series = {STOC '93}
}

@book{Kechris_2024, 
place={Cambridge}, 
series={Cambridge Tracts in Mathematics}, 
title={The Theory of Countable Borel Equivalence Relations}, 
publisher={Cambridge University Press}, 
author={Kechris, Alexander S.}, 
year={2024}, 
collection={Cambridge Tracts in Mathematics}
}

@article{EquationsInOliog,
author = {Barto, Libor and Kompatscher, Michael and Ol\v{s}\'{a}k, Miroslav and Trung, Van Pham and Pinsker, Michael},
title = {Equations in oligomorphic clones and the constraint satisfaction problem for $\omega$-categorical structures},
journal = {Journal of Mathematical Logic},
volume = {19},
number = {02},
pages = {1950010},
year = {2019},
doi = {10.1142/S0219061319500107},
}

@article{C*Algebras,
    author = {Sabok, Marcin},
    title = {Completeness of the isomorphism problem for separable C*-algebras},
    journal = {Inventiones mathematicae},
    year = {2016},
    volume = {204},
    number = {3},
    pages = {833-868},
    doi = {10.1007/s00222-015-0625-5}
}

@article{NiesSchlichtTent,
author = {Nies, Andr\'{e} and Schlicht, Philipp and Tent, Katrin},
title = {Coarse groups, and the isomorphism problem for oligomorphic groups},
journal = {Journal of Mathematical Logic},
volume = {22},
number = {01},
pages = {2150029},
year = {2022},
doi = {10.1142/S021906132150029X},
}

@article{WEISmooth,
author = {Paolini, Gianluca},
title = {The isomorphism problem for oligomorphic groups with weak elimination of imaginaries},
journal = {Bulletin of the London Mathematical Society},
volume = {56},
number = {8},
pages = {2597-2614},
doi = {https://doi.org/10.1112/blms.13086},
year = {2024}
}

@inbook{RamseyClasses, 
place={Cambridge}, 
series={London Mathematical Society Lecture Note Series}, 
title={Ramsey classes: examples and constructions}, 
booktitle={Surveys in Combinatorics 2015}, 
publisher={Cambridge University Press}, 
author={Bodirsky, Manuel}, 
year={2015}, 
pages={1–48}, 
collection={London Mathematical Society Lecture Note Series}
}

@inproceedings{MMSNP_conf,
    author = {Bodirsky, Manuel and Madelaine, Florent and Mottet, Antoine},
    title = {{A universal-algebraic proof of the complexity dichotomy for Monotone Monadic SNP}},
    pages     = {105-114},
    booktitle = {Proceedings of the 33rd Annual {ACM/IEEE} Symposium on Logic in Computer Science - {LICS}'18},
    year = {2018},
    doi = {10.1145/3209108.3209156},
    publisher = {ACM},
    xxxnote = {Preprint arXiv:1802.03255}
}

@article{MMSNP-Journal,
  author    = {Manuel Bodirsky and
               Florent R. Madelaine and
               Antoine Mottet},
  title     = {A Proof of the Algebraic Tractability Conjecture for Monotone Monadic
               {SNP}},
  longjournal = {{SIAM} Journal on Computing},
  journal = {{SIAM J. Comput.}},
  volume    = {50},
  number    = {4},
  pages     = {1359--1409},
  year      = {2021},
  doi       = {10.1137/19M128466X}
}

@article{ecsps,
  title = {The complexity of Equality Constraint Languages},
  author = {Bodirsky, Manuel and K{\'a}ra, Jan},
  longjournal = {Theory of Computing Systems},
  journal = {Theory Comput. Syst.},
  volume = {43},
  pages = {136--158},
  year = {2008},
  publisher = {Springer},
  doi = {10.1007/s00224-007-9083-9},
  note = {A conference version appeared in the proceedings of Computer Science Russia - {CSR}'06},
}

@article{BodirksyAI,
author = {Bodirsky, Manuel and Jonsson, Peter},
year = {2017},
pages = {339-385},
title = {A Model-Theoretic View on Qualitative Constraint Reasoning},
volume = {58},
journal = {Journal of Artificial Intelligence Research},
doi = {10.1613/jair.5260}
}

@article{TorsionFreeGroups,
    author = {Gianluca Paolini and Saharon Shelah},
    title = {Torsion-free abelian groups are Borel complete},
    journal = {Annals of Mathematics},
    year = {2024},
    volume = {199},
    issue = {3},
    pages = {1177--1224},
    doi = {https://doi.org/10.4007/annals.2024.199.3.4}
}

@article{ErgodicClassifiaction,
    author = {Matthew Foreman and Benjamin Weiss},
    title = {An anti-classification theorem for ergodic measure preserving transformations},
    journal = {J. Eur. Math. Soc.},
    year = {2004},
    volume = {6},
    number ={3},
    pages = {277--292}
}

@article{Silver,
title = {Counting the number of equivalence classes of Borel and coanalytic equivalence relations},
journal = {Annals of Mathematical Logic},
volume = {18},
number = {1},
pages = {1-28},
year = {1980},
issn = {0003-4843},
doi = {https://doi.org/10.1016/0003-4843(80)90002-9},
author = {Jack H. Silver}
}

@article{HKLBorel,
 ISSN = {08940347, 10886834},
 author = {Leo A. Harrington and Alexander Sotirios Kechris and Alain Louveau},
 journal = {Journal of the American Mathematical Society},
 number = {4},
 pages = {903--928},
 publisher = {American Mathematical Society},
 title = {A Glimm-Effros Dichotomy for Borel Equivalence Relations},
 urldate = {2026-01-14},
 volume = {3},
 year = {1990}
}

@Article{AhlbrandtZiegler,
author = {Gisela Ahlbrandt and Martin Ziegler},
title = {Quasi-finitely axiomatizable totally categorical theories},
journal = {Annals of Pure and Applied Logic}, 
volume = {30},
number = {1},
pages = {63--82}, 
year = {1986}}

@Article{Rubin,
author = {Matatyahu Rubin},
title = {On the reconstruction of $\omega$-categorical structures from their automorphism groups}, 
journal = {Proceedings of the London Mathematical Society}, 
volume = {3}, 
number = {69},
pages = {225-249}, 
year = {1994}}

@article{FellerPinsker,
author = {Feller, Roman and Pinsker, Michael},
title = {An Algebraic Proof of the Dichotomy for Graph Orientation Problems with Forbidden Tournaments},
journal = {SIAM Journal on Discrete Mathematics},
volume = {40},
number = {1},
pages = {1-31},
year = {2026},
doi = {10.1137/25M1726157},
}

@article{forbiddenTournaments,
author = {Bodirsky, Manuel and Guzm\'{a}n-Pro, Santiago},
title = {Forbidden Tournaments and the Orientation Completion Problem},
journal = {SIAM Journal on Discrete Mathematics},
volume = {39},
number = {1},
pages = {170-205},
year = {2025},
doi = {10.1137/23M1604849}
}

@article {MottetPinskerSmooth,
author = {Mottet, Antoine and Pinsker, Michael},
title = {Smooth approximations: An algebraic approach to {CSP}s over finitely bounded homogeneous structures},
year = {2024},
issue_date = {October 2024},
publisher = {Association for Computing Machinery},
address = {New York, NY, USA},
volume = {71},
number = {5},
issn = {0004-5411},
doi = {10.1145/3689207},
journal = {J. ACM},
month = oct,
articleno = {36},
numpages = {47}
}

@inproceedings{MottetPinskerSmoothConf,
author    = {Antoine Mottet and
               Michael Pinsker},
title = {Smooth Approximations and {CSP}s over finitely bounded homogeneous structures},
year = {2022},
booktitle = {Proceedings of the Symposium on Logic in Computer Science (LICS)}
}

@article{BMPP16,
	Author = {Manuel Bodirsky and Barnaby Martin and Michael Pinsker and Andr{\'{a}}s Pongr{\'{a}}cz},
	Title = {Constraint Satisfaction Problems for Reducts of Homogeneous Graphs},
	Journal = {SIAM Journal on Computing},
	volume=48,
	number=4, 
	pages="1224-1264",
	year=2019,
 doi       = {10.1137/16M1082974},
 	Note = {A conference version appeared in the Proceedings of the 43rd International Colloquium on Automata, Languages, and Programming, {ICALP} 2016, pages 119:1-119:14}
}

@article{BPP-projective-homomorphisms,
author={Manuel Bodirsky and Michael Pinsker and Andr\'{a}s Pongr\'acz},
title={Projective clone homomorphisms},
journal={Journal of Symbolic Logic},
volume= 86,
number=1,
pages={148-161},
year= 2021
}

@article{BPT-decidability-of-definability,
  author = {Manuel Bodirsky and Michael Pinsker and Todor Tsankov},
  title = {Decidability of definability},
  journal = {Journal of Symbolic Logic},
  volume=78,
  number=4,
  pages={1036-1054},
  year = {2013},
  note={A conference version appeared in the Proceedings of
the Twenty-Sixth Annual IEEE Symposium on Logic in Computer Science (LICS 2011), pages {321-328}}
}

@article{BodPin-Schaefer-both,
  author = {Manuel Bodirsky and Michael Pinsker},
title = {Schaefer's theorem for graphs},  
journal = {Journal of the ACM},
volume=62,
number=3,
  doi       = {10.1145/2764899},
year = {2015}, 
pages={52 pages (article number 19)},
note={A conference version appeared in the Proceedings of STOC 2011, pages 655-664}
}

@inproceedings{Bodirsky-Mottet,
  author    = {Manuel Bodirsky and
               Antoine Mottet},
  title     = {Reducts of finitely bounded homogeneous structures, and lifting tractability from finite-domain constraint satisfaction},
  booktitle = {Proceedings of the 31st Annual IEEE Symposium on Logic in Computer Science (LICS)}, 
  note = {Preprint available at ArXiv:1601.04520},
    pages     = {623-632},
  year      = {2016}
}

@Book{Book,
author = {Manuel Bodirsky},
title  = {Complexity of Infinite-Domain Constraint Satisfaction},
year   = {2021},
doi = {10.1017/9781107337534}, 
address = {Cambridge, United Kingdom; New York, NY}, 
publisher = {Cambridge University Press},
series = {Lecture Notes in Logic (52)}
}

@inproceedings{BulatovFVConjecture,
  author    = {Andrei A. Bulatov},
  title     = {A Dichotomy Theorem for Nonuniform {CSP}s},
  booktitle = {58th {IEEE} Annual Symposium on Foundations of Computer Science, {FOCS}
               2017, {B}erkeley, {CA}, {USA}, {O}ctober 15-17},
  pages     = {319-330},
  year      = {2017}
}

@InProceedings{BodirskyGrohe,
   author = {Manuel Bodirsky and Martin Grohe},
   booktitle = {ICALP},
   title = {Non-dichotomies in Constraint Satisfaction Complexity},
   year = {2008},
   publisher = {Springer Verlag},
   series = {LNCS},
   editor = {Luca Aceto and Ivan Damgard and Leslie Ann Goldberg and Magn\'us M. Halld\'orsson and Anna Ing\'olfsd\'ottir and Igor Walukiewicz},
   pages = {184 -196}
}

@Article{Cores-journal,
   author =       "Manuel Bodirsky",
   title = "Cores of countably categorical structures",
   journal = "Logical Methods in Computer Science ({LMCS})",
   volume    = {3},
   pages = {1-16},
  number    = {1},
   year = {2007}}

@inproceedings{GJKMP-conf,
author={Pierre Gillibert and Julius Jonu\v{s}as and Michael Kompatscher and Antoine Mottet and Michael Pinsker},
title={Hrushovski's encoding and $\omega$-categorical {CSP} monsters},
  volume    = {168},
  pages     = {131:1-131:17},
    publisher = {Schloss Dagstuhl - Leibniz-Zentrum f{\"{u}}r Informatik},
  year      = {2020},
  booktitle = {47th International Colloquium on Automata, Languages, and Programming (ICALP 2020)},
  series    = {LIPIcs},
}

@BOOK{Hodges,
author = {Wilfrid Hodges},
title = {A shorter model theory},
publisher = {Cambridge University Press},
address = {Cambridge},
year = {1997}}

@book {Kechris,
   AUTHOR = {Alexander Kechris},
    TITLE = {Classical descriptive set theory},
   SERIES = {Graduate Texts in Mathematics},
   VOLUME = {156},
 PUBLISHER = {Springer},
     YEAR = {1995},
   }

@article{MottetPinskerCores,
  author    = {Antoine Mottet and
               Michael Pinsker},
  title     = {Cores over {Ramsey} structures},
  journal   = {Journal of Symbolic Logic},
  volume    =   86,
  number=1,
  pages={352-361},
  year=2021
}

@article{Topo-Birk,
  author = {Manuel Bodirsky and Michael Pinsker},
  title = {Topological {B}irkhoff},
  journal={Transactions of the American Mathematical Society},
  year = {2015}, 
  volume = {367},
  pages = {2527-2549}
}

@article{Zhuk17,
  title={A Proof of {CSP} Dichotomy Conjecture},
  author={Dmitriy Zhuk},
  journal={2017 IEEE 58th Annual Symposium on Foundations of Computer Science (FOCS)},
  year={2017},
  pages={331-342},
}

@article{Zhuk20,
  author    = {Dmitriy Zhuk},
  title     = {A Proof of the {CSP} Dichotomy Conjecture},
  journal   = {Journal of the {ACM}},
  volume    = {67},
  number    = {5},
  pages     = {30:1--30:78},
  year      = {2020},
  doi       = {10.1145/3402029}
}

@article{posetCSP18,
  author    = {Michael Kompatscher and
               Trung Van Pham},
  title     = {A Complexity Dichotomy for Poset Constraint Satisfaction},
  journal   = {IfCoLog Journal of Logics and their Applications ({FLAP})},
  volume    = {5},
  number    = {8},
  pages     = {1663-1696},
  year      = {2018}
}

@Article{tcsps-journal,
author = {Manuel Bodirsky and Jan K\'ara},
title = {The Complexity of Temporal Constraint Satisfaction Problems},
journal = {Journal of the ACM},
volume = {57},
number = {2},
pages = {1-41},
doi       = {10.1145/1667053.1667058},
note = {An extended abstract appeared in the Proceedings of the Symposium on Theory of Computing (STOC)}, 
year = {2010}}

@InProceedings{AntoineZenoPaper,
  author =	{Bitter, Zeno and Mottet, Antoine},
  title =	{{Generalized Completion Problems with Forbidden Tournaments}},
  booktitle =	{49th International Symposium on Mathematical Foundations of Computer Science (MFCS 2024)},
  pages =	{28:1--28:17},
  series =	{Leibniz International Proceedings in Informatics (LIPIcs)},
  ISBN =	{978-3-95977-335-5},
  ISSN =	{1868-8969},
  year =	{2024},
  volume =	{306},
  editor =	{Kr\'{a}lovi\v{c}, Rastislav and Ku\v{c}era, Anton{\'\i}n},
  publisher =	{Schloss Dagstuhl -- Leibniz-Zentrum f{\"u}r Informatik},
  address =	{Dagstuhl, Germany},
  doi =		{10.4230/LIPIcs.MFCS.2024.28}
}

@Article{BodirskyNesetrilJLC,
author = {Manuel Bodirsky and Jaroslav Ne\v{s}et\v{r}il},
title  = {Constraint Satisfaction with Countable Homogeneous Templates},
journal = {Journal of Logic and Computation},
volume = {16},
number = {3},
doi    = {10.1093/logcom/exi083},
pages = {359-373},
year   = {2006}}

@article{BodirskyJunker,
	author = {Bodirsky, Manuel and Junker, Markus},
	doi = {10.1007/s00012-011-0110-y},
	id = {Bodirsky2010},
	journal = {Algebra universalis},
	number = {3},
	pages = {403--417},
	title = {$\aleph_0$-categorical structures: endomorphisms and interpretations},
	url = {https://doi.org/10.1007/s00012-011-0110-y},
	volume = {64},
	year = {2010},
}

@misc{PaoliniNies,
      title={Oligomorphic groups, their automorphism groups, and the complexity of their isomorphism}, 
      author={Gianluca Paolini and Andr\'{e} Nies},
      year={2025},
      eprint={2410.02248},
      archivePrefix={arXiv},
      primaryClass={math.LO},
      url={https://arxiv.org/abs/2410.02248}, 
}

@misc{MarimonPinskerReconstruction,
      title={A guide to topological reconstruction on endomorphism monoids and polymorphism clones}, 
      author={Paolo Marimon and Michael Pinsker},
      year={2025},
      eprint={2512.01086},
      archivePrefix={arXiv},
      primaryClass={math.LO},
      url={https://arxiv.org/abs/2512.01086}, 
      doi={10.48550/arXiv.2512.01086},
      note = {Preprint available from arXiv:2512.01086}
}

@article{BPPReconstructingClone,
author = {Manuel Bodirsky and Michael Pinsker and Andr\'{a}s Pongr\'acz},
title = {Reconstructing the Topology of Clones}, 
journal = {Transactions of the American Mathematical Society},
volume=369, 
pages={3707-3740},
year={2017}
}

@article{BehrischGarcia,
  title={On a stronger reconstruction notion for monoids and clones},
  author={Mike Behrisch and Edith Vargas-Garc{\'i}a},
  journal={Forum Mathematicum},
  year={2021},
  volume={33},
  pages={1487 - 1506},
  url={https://api.semanticscholar.org/CorpusID:119173700}
}

@InProceedings{removingAlgebraicity,
  author =	{Pinsker, Michael and Rydval, Jakub and Sch\"{o}bi, Moritz and Spiess, Christoph},
  title =	{{Three Fundamental Questions in Modern Infinite-Domain Constraint Satisfaction}},
  booktitle =	{50th International Symposium on Mathematical Foundations of Computer Science (MFCS 2025)},
  pages =	{83:1--83:20},
  series =	{Leibniz International Proceedings in Informatics (LIPIcs)},
  year =	{2025},
  volume =	{345},
  publisher =	{Schloss Dagstuhl -- Leibniz-Zentrum f{\"u}r Informatik},
  address =	{Dagstuhl, Germany},
  doi =		{10.4230/LIPIcs.MFCS.2025.83}
}

\end{document}